\documentclass[12pt,epsfig,amsfonts]{amsart}
\usepackage{amsmath,amsthm,amssymb,amscd,epsfig,color}
\usepackage{graphicx}

\usepackage{url}
\usepackage[all]{xy}
\usepackage{mathrsfs}
\usepackage{ulem}
\usepackage{pb-diagram} 
\usepackage{bbm} 
\newtheorem*{thmA}{Theorem A}
\newtheorem*{thmB}{Theorem B}

\newtheorem{prop}{Proposition}[section]
\newtheorem{lemma}[prop]{Lemma}

\newtheorem{thm}[prop]{Theorem}

\theoremstyle{definition}

\theoremstyle{remark}
\newtheorem{remark}[prop]{Remark}

\numberwithin{equation}{section}

\begin{document}

\author{Mao Shinoda, Hiroki Takahasi, Kenichiro Yamamoto}

\address{Department of Mathematics,
Ochanomizu University, 2-1-1 Otsuka, Bunkyo-ku, Tokyo, 112-8610, JAPAN} 
\email{shinoda.mao@ocha.ac.jp}

\address{Department of Mathematics,
Keio University, Yokohama,
223-8522, JAPAN} 
\email{hiroki@math.keio.ac.jp}

\address{Department of General Education, Nagaoka University of Technology, Nagaoka 940-2188, JAPAN}
\email{k\_yamamoto@vos.nagaokaut.ac.jp}
%{Primary 37D35; Secondary 37B10}\subjclass[2020]{Primary 37D35; Secondary 37B10, 37D30}

%\thanks{}

%\date{}

\title[Ergodic optimization for continuous functions]
 {Ergodic optimization for continuous functions on the Dyck-Motzkin shifts}

\begin{abstract}
Ergodic optimization aims to describe dynamically invariant probability measures that maximize the integral of a given function. 
The Dyck and the Motzkin shifts are well-known examples of transitive subshifts that are not intrinsically ergodic.
We show that
the space of continuous functions on any Dyck-Motzkin shift splits into two subsets: one is a dense $G_\delta$ set with empty interior for which any maximizing measure has zero entropy; the other is contained in the closure of the set of functions having uncountably many, fully supported measures that are Bernoulli.
One key ingredient of a proof of this result is the path connectedness of the space of
ergodic measures of the Dyck-Motzkin shift.
\end{abstract}

\maketitle
%\tableofcontents

 \section{Introduction}
  Ergodic optimization aims to describe 
 properties of dynamically invariant
maximizing measures. In its most basic form, main constituent components are:
a continuous map $T$ of a compact metric space $X$; 
the space $M(X,T)$ of $T$-invariant Borel probability measures endowed with the weak* topology together with the space $M^{\rm e}(X,T)$ of its ergodic elements;
a continuous function $f\colon X\to\mathbb R$. Elements of $M(X,T)$ that attain the supremum \begin{equation}\label{alphaT}\Lambda_T(f)=\sup\left\{\int f{\rm d}\mu\colon\mu\in M(X,T)\right\}\end{equation}
are called 
{\it $f$-maximizing measures}.
The set of $f$-maximizing measures, denoted by $M_{\rm max}(f)$, is non-empty and contains elements of
$M^{\rm e}(X,T)$.
%\textcolor{red}{why $M^{\rm e}(X,T)$ is a Borel subset of $M(X,T)$?}
For a given $(X,T)$ and a Banach space of real-valued functions on $X$, we aim to establish properties of 
elements of $M_{\rm max}(f)$
for a `typical' function $f$ in the space.
The regularity of functions is crucial. For $(X,T)$ with some %sort of 
expanding or hyperbolic behavior and a H\"older continuous $f$,
%of continuous ones, the maximizing measure is unique and supported on a periodic orbit, see \cite{BoZha15,Bou01,Con16,CLT01,QS12}
the Ma\~n\'e-Conze-Guivarc'h lemma %restricts 
characterizes $f$-maximizing measures via their supports \cite{Bou00,Bou01,CLT01,Mor09}.
The analysis of functions in the space 
$C(X)$ of real-valued continuous functions on $X$ endowed with the supremum norm $\|\cdot\|_{C^0}$ is completely different: the Ma\~n\'e-Conze-Guivarc'h lemma is no longer valid, but duality arguments are available
and can be used to prove the occurrence of pathological phenomena.

%for a Baire generic continuous function, the maximizing measure is unique \cite{Bou01,Jen06}, has zero entropy \cite{Bre08,Mor10}, but fully supported \cite{Bou01}.
 Morris \cite[Corollary~2]{Mor10} proved that for $(X,T)$ with Bowen's specification property \cite{Bow71}, 
 the maximizing measure is unique, fully supported on $X$ (charging any non-empty open subset of $X$), has zero entropy, and is not strongly mixing for a generic continuous function, thereby unifying the the result of Bousch and Jenkinson \cite{BouJen02} and that of Br\'emont \cite{Bre08}. 
 In contrast to Morris's result, 
Shinoda \cite[Theorem~A]{Shi18} proved that 
for a dense set of continuous functions on
a topologically mixing Markov shift (subshift of finite type),
 there exist uncountably many, fully supported ergodic maximizing measures with positive entropy, which are actually Bernoulli. For an analogous result 
 on expanding Markov interval maps with holes, see \cite{ShiTak20}.

As a generalization and unification of
the result of Morris \cite[Corollary~2]{Mor10} and that of Shinoda \cite[Theorem~A]{Shi18} concerning entropies and supports of maximizing measures,
in \cite{STY24}
 the authors proved the following statement
for a wide class of non-Markov subshifts $\Sigma$ over a finite alphabet:
There exists a constant $h_{\rm spec}^\bot(\Sigma)\in[0,h_{\rm top}(\Sigma))$
 such that
if $h_{\rm spec}^\bot(\Sigma)\leq H<h_{\rm top}(\Sigma)$ then
\begin{itemize}
\item[(I)] the set
$\{f\in C(\Sigma)\colon h(\mu)\leq H
\text{ for all $\mu\in M_{\rm max}(f)$}\}$
is dense $G_\delta$;  
\item[(II)] for any $f\in C(\Sigma)$ not contained in the dense $G_\delta$ set in (I)
 and for any neighborhood $U$ of $f$ in $C(\Sigma)$, there exists $g\in U$ 
 such that 
 $\{\mu\in M_{\rm max}(g)\colon h(\mu)>H\}$
contains uncountably many fully supported ergodic measures,
\end{itemize}
where $h_{\rm top}(\Sigma)<\infty$ denotes the topological entropy of $\Sigma$, and 
 $h(\mu)\in[0,h_{\rm top}(\Sigma)]$ denotes the measure-theoretic entropy of a shift-invariant measure $\mu$ on $\Sigma$. 
 The constant $h_{\rm spec}^\bot(\Sigma)$ is called {\it the obstruction entropy to specification} \cite{CT14}.
Recall that dense $G_\delta$ sets are countable intersections of open dense subsets, and a property that holds for a dense $G_\delta$ set is said to be generic.

A subshift carrying a unique measure of maximal entropy is called {\it intrinsically ergodic}.
The statement in the previous paragraph requires $h_{\rm spec}^\bot(\Sigma)<h_{\rm top}(\Sigma)$ that actually implies the intrinsic ergodicity of the subshift $\Sigma$ \cite[Theorem~C]{CT13}.
  % for a Baire generic continuous function, the maximizing measure is unique \cite{Jen06}, fully supported \cite{Bou01}, has zero entropy \cite{Bre08}.
% \begin{thm}[Bousch \cite{Bou01}]\label{Bousch} If $(X,d)$ is a compact metric space and $T\colon X\to X$ is an expanding map, then there is a residual subset $\mathcal O$ of $C(X)$ such that for all $\varphi\in\mathcal O$ the $\varphi$-maximizing measure is unique and fully supported on $X$. \end{thm}
% We say a Borel probability measure $\mu$ on $X$ is {\bf fully supported} on $X$ if it charges any non-empty open subset of $X$.
%Although only one-sided Markov shifts were treated in \cite{Shi18}, the same proof as in \cite{Shi18} works for two-sided Markov shifts.
%In \cite{Shi18}, Shinoda only stated that the set of ergodic $f$-maximizing measures for $f\in F$ is uncountable, but a glance at the proof shows the statement on the cardinality as in Theorem~\ref{Shinoda-A}
Therefore, it is natural to ask if 
 an analogous statement %of the above kind 
 holds for a subshift that is not intrinsically ergodic.

As a counterexample to the conjecture of Weiss \cite{W70}, 
Krieger \cite{Kri74} proved that the Dyck shift
 has exactly two ergodic measures of maximal entropy. % which are fully supported and Bernoulli. 
%There is a universal formal language due to W. Dyck. The Dyck shift is the symbolic dynamics generated by that language.
%The Dyck shift is a rich source of interesting phenomena in symbolic dynamics. These measures are exponentially mixing for H\"older functions \cite{T23}, and so satisfy the central limit theorem. For other interesting properties of the Dyck shift, see \cite{HI05,K91,Mey08} for example. 
The Motzkin shift \cite{M04,M05} is a subshift determined by the Dyck shift and the units.
Krieger \cite{Kri00} introduced a certain class of shift spaces having some algebraic property, called property A subshifts. The Dyck and Motzkin shifts are prototypes %fundamental shift spaces 
in this class, and
rich sources of interesting phenomena different from those in Markov shifts, %rich sources of interesting phenomena in symbolic dynamics as well as in smooth dynamics, 
see 
\cite{HI05,M04,M05,Mey08,T23}
%\cite{HI05,K91,M04,M05,Mey08,STY21,STY23,STYY,T23,TY23} 
for example.
%Meyerovich investigated tail invariant measures of the Dyck shift.
%The second named author proved that the two ergodic measures of maximal entropy of the Dyck shift are exponentially mixing for H\"older functions, and so satisfy the central limit theorem. 
In this paper we consider ergodic optimization for continuous functions on these subshifts that are not intrinsically ergodic.

%A {\it Bernoulli measure} is a  Borel probability measure on the one- or two-sided full shift space on $k$-symbols $(k\geq2)$ that is determined by a non-degenerate $k$-dimensional probability vector $(p_1,\ldots,p_k)$. Bernoulli measures are shift invariant and ergodic, and define Bernoulli systems.
%We say a measurable dynamical system $(T,\nu)$, or simply $\nu$ is $(p_1,\ldots,p_k)$-{\it Bernoulli} if it is isomorphic to the $(p_1,\ldots,p_k)$-Bernoulli system. 

%\item The non-commutative entropy is strictly larger than the topological entropy \cite{M04}.

%The Dyck-Motzkin shift is a slight generalization of the Dyck shift.
%For all these and other interesting recent developments, %including \cite{TT},  there are still many unsolved problems on the dynamics of the Dyck shift.

\subsection{Ergodic optimization for continuous functions}\label{erg-o}
In Section~\ref{dyck} we will introduce shift spaces $\Sigma_D$
 with two non-negative integer parameters $(M,N)$, called the {\it Dyck-Motzkin shifts}.
 The $(M,0)$ Dyck-Motzkin shift is nothing but the Dyck shift on $2M$ symbols,
consisting of $M$ brackets, left and right in pair, whose admissible words are words of legally aligned brackets.
 The $(M,N)$ Dyck-Motzkin shift with $N\neq0$ is nothing but the Motzkin shift on $2M+N$ symbols, consisting of the $M$ brackets and $N$ units whose admissible words are words of legally aligned brackets with freely interspersed  units.
 Our main result below recovers the above (I) (II) with $H=0$ for the Dyck-Motzkin shifts.

\begin{thmA}\label{thma}
Let $\Sigma_D$ be a Dyck-Motzkin shift.
\begin{itemize}
\item[(a)] The set
$\{f\in C(\Sigma_D)\colon h(\mu)=0
\text{ for all $\mu\in M_{\rm max}(f)$}\}$
is dense $G_\delta$.
\item[(b)] There exists a dense subset $\mathscr{D}$ of $C(\Sigma_{D})$ such that for any $f\in \mathscr{D}$, 
$M_{\rm max}(f)$
contains uncountably many elements that are fully supported on $\Sigma_D$ and Bernoulli.
\end{itemize}
\end{thmA}
%An analogous statement holds for any one-sided Dyck-Motzkin shift. For simplicity, we restrict ourselves to the two-sided shifts.

%It is worthwhile to compare Theorem~A and \cite[Theorem~B]{STY24}.
Statements like Theorem~A replacing `Bernoulli' in (b) by `ergodic, positive entropy' were obtained in \cite[Theorem~B]{STY24} for any subshift $\Sigma$ %over a finite alphabet
for which $h_{\rm top}(\Sigma)>h_{\rm spec}^\bot(\Sigma)=0$ and ergodic measures are entropy dense. The definition of entropy density can be found in \cite{EKW94}. 
For the $(M,N)$ Dyck-Motzkin shift $\Sigma_D$, we note that $h_{\rm top}(\Sigma_D)=h_{\rm spec}^\bot(\Sigma_D)=\log(M+N+1)$ holds, and
 ergodic measures are not entropy dense.

For proofs of (a) and (b) of Theorem~A, we develop 
ideas of Morris \cite[Theorem~1.1,\ Corollary~2]{Mor10} and Shinoda \cite[Theorem~A]{Shi18} respectively, both related to approximations of ergodic measures in the weak* topology.
Regarding (a), a key observation is that
 Bowen's specification property does not hold
 for the Dyck-Motzkin shift, but the hypothesis of
  Bowen's specification property in \cite[Corollary~2]{Mor10} can actually be weakened to the density of invariant measures of zero entropy in the space of ergodic measures.
   For any Dyck-Motzkin shift,
  we show in Section~\ref{approx-sec} that 
 $CO$-measures (shift-invariant ergodic measures supported on  periodic orbits)
 are dense in the space of ergodic measures. This property allows us to slightly modify the argument in the proof of \cite[Theorem~1.1]{Mor10} to conclude Theorem~A(a).

The proof of Theorem~A(b)
 deserves a special attention as it gives a new insight into the structure of the spaces of ergodic measures of the Dyck-Motzkin shifts.
Below we give further explanations, but first require simple definitions.
Let $X$ be a topological space and let $x$, $y\in X$ be distinct points.
A continuous map $p\colon[0,1]\to X$ such that $p(0)=x$ and $p(1)=y$ is called a {\it path} joining $x$, $y$. We say a path $p\colon[0,1]\to X$ {\it lies} in $Y\subset X$ if $p(t)\in Y$ holds for all $t\in [0,1]$.

Israel \cite[Section~V]{Isr79} proved an approximation theorem about tangent functionals to convex functions, and used it  for lattice systems in statistical mechanics to prove the existence of a dense set of continuous interactions that admit  uncountably many ergodic equilibrium states. 
Shinoda's proof of \cite[Theorem~A]{Shi18}
is an adaptation of Israel's argument to ergodic optimization that is briefly outlined as follows.
For a real-valued continuous function $f_0$ on a compact metric space $X$ and a continuous map $T\colon X\to X$ such that $M^{\rm e}(X,T)$ is arcwise connected, 
%Shinoda \cite{Shi18} noted that one can define a natural non-atomic Borel probability measure on $M^{\rm e}(X,T)$. 
%and $f_0\in C(X)$, 
she took an ergodic measure $\mu\in M_{\rm max}(f_0)$ and a path $t\in[0,1]\mapsto \mu_t\in M^{\rm e}(X,T)$ 
such that $\mu_0=\mu$, and then used the version of the Bishop-Phelps theorem \cite[Theorem~V.1.1]{Isr79}
to approximate $f_0$ by $f\in C(X)$ so that $\{\mu_t\colon t\in[0,1]\}$ contains uncountably many  elements of $M_{\rm max}(f)$.
%For a general result obtained with this idea, see \cite[Theorem~B]{Shi18}.
One can control properties of the maximizing measures by carefully choosing the path.
To choose a path such that the uncountably many maximizing measures are fully supported and Bernoulli, 
Shinoda used Sigmund's result \cite[Theorem~B]{Sig77} which asserts that the space of shift-invariant ergodic mesures on a topologically mixing Markov shift is path connected.

Our strategy for the proof of Theorem~A(b) is to substantially extend Shinoda's path argument to the Dyck-Motzkin shifts. 
 Since the Dyck-Motzkin shifts are not Markov, 
 Sigmund's result \cite[Theorem~B]{Sig77} is no longer valid. To overcome this difficulty, we delve into the structure of the shift space and show
%it suffices to show that 
%for any shift-invariant ergodic measure $\mu$ 
%of the Dyck-Motzkin shift, there is
% a path $t\in[0,1]\mapsto \mu_t$ that lies in the space of shift-invariant ergodic measures and satisfies $\mu_0=\mu$. 
 the following abundance of paths of ergodic measures of high complexity:
\smallskip

 \begin{itemize}
\item[(1)] {\it Any pair of ergodic measures of the Dyck-Motzkin shift in any weak* open ball can be joined by a path that lies in that ball}, 
and moreover

\item[(2)]  {\it this path `almost' lies in the set of measures that are fully supported and Bernoulli}. 
\end{itemize}

Since any Dyck-Motzkin shift contains many subshifts of finite type (SFTs) in its shift space, 
Sigmund's result \cite[Theorem~B]{Sig77} can still be used to find a high complexity path joining two ergodic measures for the Dyck-Motzkin shift that are supported on the same properly embedded SFT.
Proper embedding means a one-to-one, into but not onto conjugacy of shift spaces.
It is not always possible to find an SFT
that supports a given pair of ergodic measures.
 Hamachi and Inoue \cite[Theorem~5.3]{HI05} provided a necessary and sufficient condition for the existence of a proper embedding of an irreducible SFT into the Dyck shift in terms of topological entropy of the shift spaces and multipliers of periodic points in them.
From their result it immediately follows that for any integer $M\geq2$, there is no SFT of entropy $\log(M+1)$ that can be embedded into the Dyck shift on $2M$ symbols.
For a corresponding result on the property A subshifts, see \cite[Theorem~5.8]{HIK09}.
In particular, there is no embedded SFT in the Dyck shift that supports the two ergodic measures of maximal entropy $\log(M+1)$ constructed by Krieger \cite{Kri74}.
An analogous statement holds for the Motzkin shift.

%For the Dyck shift on $2M$ symbols,
%Krieger \cite{Kri74} proved the existence of exactly two ergodic measures of maximal entropy 
%$\log(M+1)$. 
Krieger's construction of the two ergodic measures of maximal entropy for the Dyck shift 
relies on the construction of two different Borel embeddings of the full shift on $M+1$ symbols into the Dyck shift. Borel embedding means a %homeomorphic
one-to-one
conjugacy defined on a shift-invariant proper Borel subset
that does not have a continuous extension to the whole shift space (see Section~\ref{Borel-sect} for the definition).
In order to prove (1) (2), we begin by extending Krieger's construction \cite[Section~4]{Kri74} of Borel embeddings of the full shifts to the Dyck-Motzkin shift $\Sigma_D$. We then transport sequences of ergodic measures on $\Sigma_D$ to the two full shift spaces via the inverses of these Borel embeddings, and 
construct paths joining the transported measures.
Finally, we 
transport these paths back to the space of ergodic measures on $\Sigma_D$, and concatenate the  transported paths to obtain a desired path.
Since the Borel embeddings do not have
%are not defined on 
continuous extensions to the whole full shift spaces,
the transport in the last step
%is not immediate and 
needs a justification.
For precise statements of (1) (2) with relevant definitions and more detailed explanations, see Section~\ref{abundance}.

\subsection{Connectedness of spaces of ergodic measures}
 Recall that a topological space $X$
 is {\it path connected} if for any pair $x,y$ of its distinct points there exists a path that lies in $X$ and joins $x$, $y$. 
 We say $X$ is {\it arcwise connected} if it is path connected and the path can be taken to be a homemorphism onto its image.
We say $X$ is {\it locally path connected} (resp. {\it locally arcwise connected}) if 
 any point in it has a neighborhood base consisting of open sets that are path connected (resp. arcwise connected).
 For a Hausdorff space, 
 the path connectedness implies the arcwise connectedness.
% \url{https://math.stackexchange.com/questions/3488915/arcwise-connected-vs-path-connected}

The property (1) immediately yields the following result.

\begin{thmB}
The space of shift-invariant ergodic Borel probability measures on any Dyck-Motzkin shift is path connected and locally path connected with respect to the weak* topology.\end{thmB}

      An investigation of the topological structure of the space $M(X,T)$ 
      began with the works of Sigmund \cite{Sig70,Sig77} in the 70s, and has recently gained a renewed impetus. 
      One motivation comes from the fact that % (From \cite{KKK18}) 
      every Polish topological space is homeomorphic to a set of ergodic measures of some shift space over a finite alphabet, 
which follows from \cite[Theorem]{Hay75} and \cite[Theorem~5]{Dow91}.
%\textcolor{red}{Check if \cite{Cho69} is valid here.}
The space $M(X,T)$ is a
  Choquet simplex %\cite{Ph01} 
      whose extreme points are precisely the set
      $M^{\rm e}(X,T)$ of ergodic measures. 
            Choquet simplices with dense extreme points are isomorphic up to affine homeomorphisms, and the unique simplex is called the Poulsen simplex \cite{Pou61}. If $M(X,T)$ is a Poulsen simplex, then the path connectedness and the local path connectedness of $M^{\rm e}(X,T)$ follow from a complete description of the topological structure of the Poulsen simplex given in \cite{LOS78}.
            
      The space of ergodic measures of
any Dyck-Motzkin shift is 
not Poulsen, and it is path connected by Theorem~B. In his blog, Climenhaga gave a rough sketch of a proof of the path connectedness of the space of ergodic measures of the Dyck shift. 
However a justification is needed.
      For other examples of subshifts 
      whose spaces of ergodic measures are not Poulsen and path connected, see \cite[Corollary~7]{KKK18} and \cite[Section~4]{KLW16}.
For relevant results on structures of the spaces of ergodic measures of partially hyperbolic diffeomorphisms, see
   \cite{DGMR19,GP17}.
   A sufficient condition for the path connectedness of $M^{\rm e}(X,T)$ in terms of periodic points and $CO$-measures of $T$ was given in \cite[Theorem~6.1]{GK18}, which however does not apply to the Dyck-Motzkin shift.
In any of these previous works, 
there was no discussion on the local path connectedness of the space of ergodic measures. Needless to say,
a proof of the local path connectedness is more delicate than that of the mere path connectedness.

%A different approach to Sigmund's theorem is to show that  ergodic measures on transitive topological Markov shifts are dense in the space of all invariant measures. Since the latter space is a Choquet simplex and ergodic measures are its extremal points, it means that this space is the Poulsen simplex (which is unique up to an affine homeomorphism). The desired result now follows from a complete description of the Poulsen simplex given in \cite{LOS78}.

     The rest of this paper consists of three sections. In Section~2 we collect preliminary results needed for the proofs of our main results. In Section~3 we give precise statements of (1) (2) and prove them. In Section~4 we complete the proofs of Theorems~A and B.

\section{Preliminaries}
In this section we collect preliminary results needed for the proofs of our main results.
After recalling
 basic terms and notation in symbolic dynamics
 in Section~\ref{s-w-l}, 
 we introduce the Dyck-Motzkin shifts in Section~\ref{dyck}. 
 Following Krieger \cite[Section~4]{Kri74},
in Section~\ref{iso-sec} we classify ergodic measures of the Dyck-Motzkin shifts into three types, and in Section~\ref{DM-str} construct embeddings of two full shifts on $M+N+1$ symbols into the $(M,N)$ Dyck-Motzkin shift. In Section~\ref{transport-esc} we deal with a transportation of ergodic measures by means of these embeddings. 
In Section~\ref{Borel-sect} we show that the embeddings constructed in Section~\ref{DM-str} are Borel embeddings.
%is an addendum to Section~\ref{DM-str}.
In Section~\ref{approx-sec} we prove an approximation result by $CO$-measures for the Dyck-Motzkin shifts.
%We say $X$ is {\bf path connected} if every pair of distinct points of $X$ is joined by a path. We say $X$ is {\bf arcwise connected} if every pair of distinct points in $X$ is joined by a path $p\colon[0,1]\to X$ that is a  homeomorphism onto its image. Clearly the arcwise connectedness implies the path connectedness.
%If $X$ is Hausdorff, the path connectedness implies the arcwise connectedness.\url{https://math.stackexchange.com/questions/3488915/arcwise-connected-vs-path-connected}  We say $X$ is {\bf locally arcwise connected} if  any point in $X$ has a neighborhood base consisting of arcwise connected open sets.
 %for any $x\in X$ and any neighborhood $U$ of $x$ there exists an arcwise connected open set $V$ in $X$ such that $x\in V\subset U$. \url{https://en.wikipedia.org/wiki/Locally_connected_space}\url{https://en.wikipedia.org/wiki/Lexicographic_order_topology_on_the_unit_square}\url{https://en.wikipedia.org/wiki/Order_topology}

\subsection{Basic terms and notation}\label{s-w-l}
Let $S$ be a non-empty finite discrete set, called a {\it finite alphabet} and let $S^{\mathbb Z}$ denote 
the two-sided Cartesian product topological space of $S$. %This topology is metrizable with the Hamming metric. 
  The left shift acts continuously on $S^{\mathbb Z}$.
   A shift-invariant closed subset of $S^{\mathbb Z}$ is called a {\it subshift} over $S$.
  A finite string $\omega=\omega_1\omega_2\cdots \omega_n$ of elements of $S$ is called a {\it word} of length $n$ in $S$. For convenience, we introduce an empty word $\emptyset$ by the rules
  $\emptyset \omega=\omega\emptyset=\omega$ for any word $\omega$ in $S$. The word length of the empty word is set to be $0$.
 For a subshift $\Sigma$ over $S$ and $n\in\mathbb N\cup\{0\}$,
 let $\mathcal L_n(\Sigma)$
 denote the collection of words in $S$ of word length $n$ that appear in some elements of $\Sigma$.
 Put $\mathcal L(\Sigma)=\bigcup_{n\in\mathbb N\cup\{0\}}\mathcal L_n(\Sigma)$.
Words in $\mathcal L(\Sigma)\setminus\{\emptyset\}$ are called {\it admissible}. 
For a subshift $\Sigma$
%that is either $\Sigma_D$, $\Sigma_\alpha$ or $\Sigma_\beta$, 
and for $j\in\mathbb Z$, $n\in\mathbb N$, $\omega\in \mathcal L_n(\Sigma)$, 
define
\[\Sigma(j;\omega)=\{(x_i)_{i\in\mathbb Z}\in\Sigma\colon x_{i+j-1}=\omega_{i}\ \text{ for  }i=1,\ldots, n\}.\]
Unless otherwise stated, we use the letter $\sigma$ to denote the left shift acting on a subshift $\Sigma\colon(\sigma x)_i=x_{i+1}$ for all $i\in\mathbb Z$. 
For a subshift $\Sigma$, let $M(\Sigma)$ denote the space of shift-invariant Borel probability measures on $\Sigma$ endowed with the weak* topology, and let $M^{\rm e}(\Sigma)$ denote the space of elements of $M(\Sigma)$ that are ergodic.
For a sequence $\{\mu_n\}_{n\in\mathbb N}$ of Borel probability measures on $\Sigma$ that converges to $\mu\in M(\Sigma)$ in the weak* topology, we write $\mu_n\to\mu$.

Let $t\in\mathbb N$ and let $A=(a_{ij})$ be a $t\times t$ matrix all whose entries are $0$ or $1$. 
We assume that every row of $A$ has a nonzero entry.
%for each $i\in\{1,\ldots,t\}$ there is $j\in\{1,\ldots,t\}$ such that $a_{ij}=1$. 
The subshift \[\Sigma_A=\{(x_i)_{i\in\mathbb Z}\in \{1,\ldots,t\}^{\mathbb Z }\colon a_{x_ix_{i+1}}=1\text{ for all }i\in\mathbb Z\}\] is called a  {\it Markov shift} or a {\it subshift of finite type} (SFT) determined by the transition matrix $A$. In the case $a_{ij}=1$ for all $i$, $j\in\{1,\ldots,t\}$, $\Sigma_A$ is called the {\it full shift} on $t$ symbols. Write $A^n=(a_{ij}^{(n)})$ for $n\in\mathbb N$. We say $\Sigma_A$ is {\it topologically mixing} if there exists $n\in\mathbb N$ such that $a_{ij}^{(n)}\neq0$ for all $i$, $j\in\{1,\ldots,t\}$.

\subsection{The Dyck-Motzkin shift}\label{dyck}
Let $M\geq2$, $N\geq0$ be integers
and let
 \[D_\alpha=\{\alpha_1,\ldots,\alpha_M\}\ \text{ and }\ D_\beta=\{\beta_1,\ldots,\beta_M\},\]
 which are interpreted as sets of left brackets and right brackets respectively: $\alpha_i$ and $\beta_i$ are in pair for $i=1,\ldots,M$.
The $(M,N)$ Dyck-Motzkin shift is a two-sided subshift over the finite alphabet 
\[D=D_\alpha\cup D_0\cup D_\beta,\]
which consists of $2M+N$ symbols where
 $\#D_0=N$.
 If $N=0$ then $D_0$ is an empty set. If $N\neq0$ then we write
$D_0=\{1_1,\ldots,1_N\}.$
Let $D^*$ denote the set of finite words in $D$.
Consider the monoid with zero, with $2M+N$ generators in $D$ and the unit element
$1$ with relations 
\[\alpha_i\cdot\beta_j=\delta_{i,j},\
0\cdot 0=0,\ 1_k\cdot1_\ell=1\ \text{ for } i,j\in\{1,\ldots,M\},\ k,\ell\in \{1,\ldots,N\},\] \[\omega\cdot 1= 1\cdot\omega=\omega,\
 \omega\cdot 0=0\cdot\omega=0\ 
\text{ for }\omega\in D^*\cup\{ 1\},\]
 \[\omega\cdot 1_\ell= 1_\ell\cdot\omega=\omega\ 
\text{ for }\omega\in D^*\text{ and }\ell\in \{1,\ldots,N\},\]
where $\delta_{i,j}$ denotes Kronecker's delta.
For $n\in\mathbb N$ and $\omega_1\cdots\omega_n\in D^*$ 
let
\[{\rm red}(\omega_1\cdots\omega_n)=\prod_{i=1}^n\omega_i.\]
The $(M,N)$ Dyck-Motzkin shift is defined by
\[\Sigma_{D}=\{x=(x_i)_{i\in\mathbb Z}\in D^{\mathbb Z}\colon {\rm red}(x_j\cdots x_k)\neq0\ \text{ for all }j,k\in\mathbb Z\text{ with }j<k\}.\]

Another way to define the $(M,N)$ Dyck Motzkin shift $\Sigma_D$ is the following. Consider a labeled directed graph that consists of infinitely many vertices $V_{i,j}$, $i=0,1,\ldots$, $j=1,\ldots,M^i$
together with edges each labeled with a unique symbol in $D$.
Each vertex $V_{i,j}$ has 
$M$ outgoing edges

 \[V_{i,j}\stackrel{\alpha_1}{\longrightarrow}V_{i+1,Mj-M+1},\ \ldots,\  V_{i,j}\stackrel{\alpha_{M-1}}{\longrightarrow}V_{i+1,Mj-1},\ V_{i,j}\stackrel{\alpha_M}{\longrightarrow}V_{i+1,Mj}\]
 and $M$ incoming edges
 \[V_{i,j}\stackrel{\beta_1}{\longleftarrow}V_{i+1,Mj-M+1},\ \ldots,\ V_{i,j}\stackrel{\beta_{M-1}}{\longleftarrow}V_{i+1,Mj-1},\ V_{i,j}\stackrel{\beta_M}{\longleftarrow}V_{i+1,Mj}.\]
 The bottom vertex $V_{0,1}$ has additional $M$ loop edges $V_{0,1}\stackrel{\beta_1,\ldots,\beta_M}{\longrightarrow}V_{0,1}$.
If $N\neq0$, then each vertex $V_{i,j}$ has additional
  $N$ loop edges
$V_{i,j}\stackrel{1_1,\ldots,1_N}{\longrightarrow}V_{i,j}$. 
Let $\Sigma_D^+$ denote the set of one-sided infinite sequences of elements of $D$ that are associated with the infinite labeled paths in this graph starting at $V_{0,1}$. Then $\Sigma_D$ is the invertible extension of $\Sigma_D^+$: 
\[\Sigma_{D}=\{x=(x_i)_{i\in\mathbb Z}\in D^{\mathbb Z}\colon x_jx_{j+1}\cdots\in\Sigma_D^+ \text{ for all }j\in\mathbb Z\}.\]

 \begin{figure}
\begin{center}
\includegraphics[height=8.5cm,width=13cm]
{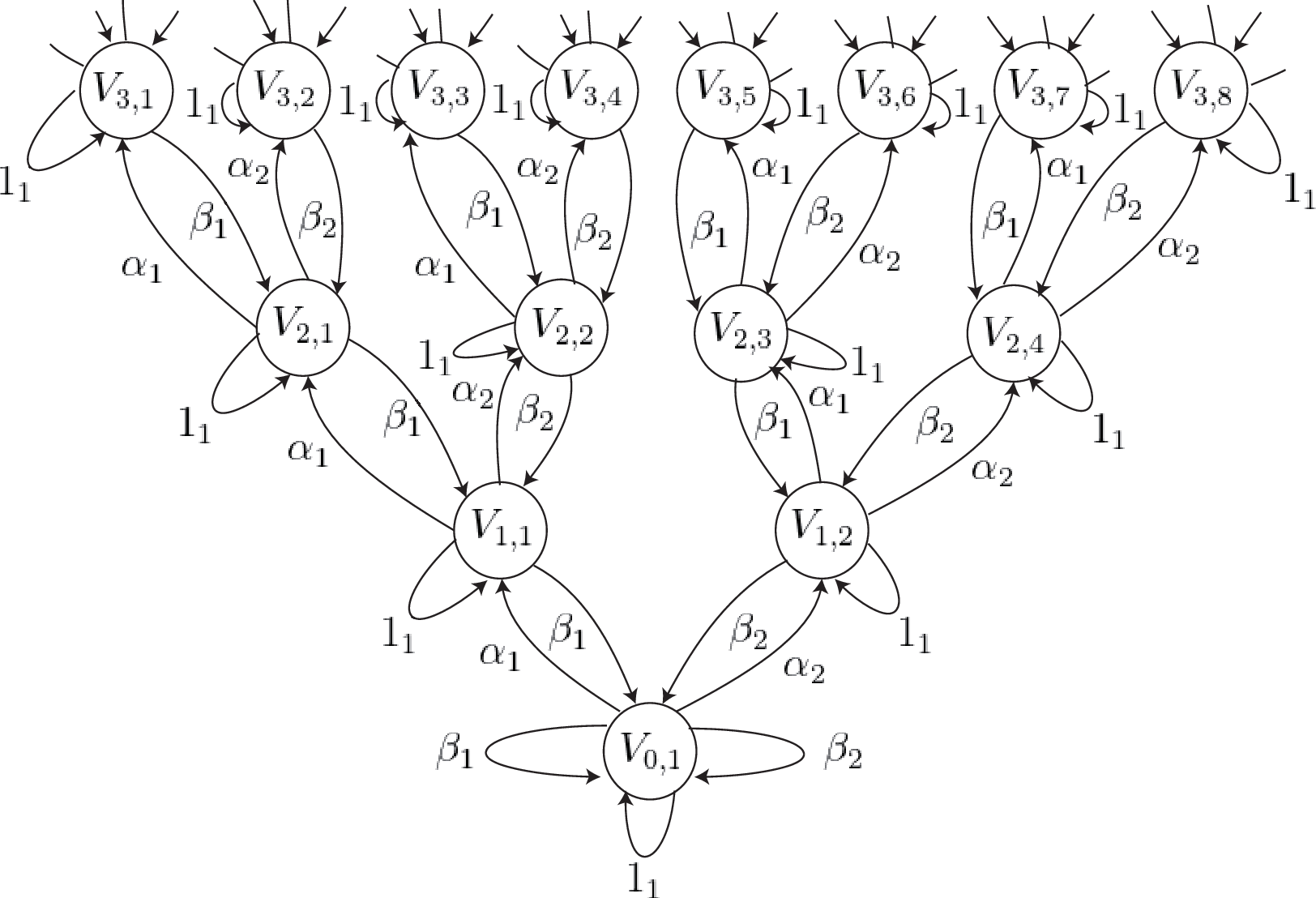}
\caption
{Part of the labeled directed graph associated with the $(2,1)$ Dyck-Motzkin shift. 
Each upward (resp. downward) edge is labeled with $\alpha_1$ or $\alpha_2$ (resp. $\beta_1$ or $\beta_2$).
Each vertex has one loop edge labeled with $1_1$.}\label{Motzkin-diagram}
\end{center}
\end{figure}

Part of the labeled directed graph associated with the $(2,1)$ Dyck-Motzkin shift is shown in \textsc{Figure}~\ref{Motzkin-diagram}. The set of admissible words coincides with the set of strings of labels associated with the finite paths in the graph starting at $V_{0,1}$.
Removing all the loop edges labeled with $1_1$ gives part of 
the labeled directed graph associated with the $(2,0)$ Dyck-Motzkin shift.
It is easy to see that  $\Sigma_D$ has infinitely many forbidden words, and so it is not a Markov shift.

    \subsection{Classification of ergodic measures }\label{iso-sec}
    For each $j\in\mathbb Z$ define $G_j\colon \Sigma_D\to\{-1,0,1\}$ by \[G_j(x)=\sum_{k=1}^{M}(\delta_{{\alpha_k},x_j}-\delta_{\beta_k,x_j}).\]
    We have $G_j(x)=1$ if $x_j\in D_\alpha$, $G_j(x)=0$ if $x_j\in D_0$, $G_j(x)=-1$ if $x_j\in D_\beta$.
    For each $i\in\mathbb Z$ define $H_i\colon \Sigma_D\to\mathbb Z$ by
      \[H_i(x)=\begin{cases}\sum_{j=0}^{i-1} G_j(x)&\text{ for }  i\geq1,\\-\sum_{j=i}^{-1} G_j(x)&\text{ for } i\leq -1,\\
      0&\text{ for }i=0.\end{cases}\]
      The function $H_i$ for $i\geq1$ (resp. $i\leq-1$) counts the difference between the number of symbols in $D_\alpha$ and that in $D_\beta$  appearing in $x_0\cdots x_{i-1}$ (resp. $x_i\cdots x_{-1}$).
For $i$, $j\in\mathbb Z$ define \[\{H_i=H_j\}=\{x\in\Sigma_D\colon H_i(x)=H_j(x)\}.\]
For $i,j\in\mathbb Z$ with $i<j$, we have
$H_i(x)=H_j(x)$ if and only if the number of symbols in $D_\alpha$ that appear in $x_i\cdots x_{j-1}$
equals the number of symbols in $D_\beta$ that appear in $x_i\cdots x_{j-1}$. 
%(So, if $N=0$ and $j-i$ is odd then $\{H_i=H_j\}=\emptyset$). 
In particular, if ${\rm red}(x_i\cdots x_{j-1})=1$ then $H_i(x)=H_j(x)$ holds. 
If $x\in\Sigma_D$,
$i<j$ and $H_i(x)=H_j(x)$ then $i+1<j$ holds. If moreover 
$x_i\in D_\alpha$ (resp. $x_{j-1}\in D_\beta$), then there exists $k\in\{i+1,\ldots,j-1\}$ 
(resp. $k\in\{i,\ldots,j-2\}$)
such that the left bracket at 
position $i$ (resp. the right bracket at position $j-1$) in $x$ is closed with the corresponding right (resp. left) bracket at position $k$ in $x$:
${\rm red}(x_i\cdots x_{k})=1$
and ${\rm red}(x_i\cdots x_{\ell})\neq1$ for all $\ell\in\{i,\ldots,k-1\}$ (resp. ${\rm red}(x_k\cdots x_{j-1})=1$
and ${\rm red}(x_\ell\cdots x_{j-1})\neq1$ for all $\ell\in\{k+1,\ldots,j-1\}$).

We introduce three pairwise disjoint shift-invariant Borel sets
    \[\label{3-sets}\begin{split}A_0&=\bigcap_{i=-\infty }^{\infty}\left(\left(\bigcup_{j=1}^\infty\{ H_{i+j}=H_i\}\right)\cap\left(\bigcup_{j=1}^\infty\{ H_{i-j}=H_i\}\right)\right),\\
A_\alpha&=\left\{x\in\Sigma_D\colon
\lim_{i\to\infty}H_i(x)
=\infty\ \text{ and } \ \lim_{i\to-\infty}H_i(x)=-\infty\right\},\\
A_\beta&=\left\{x\in\Sigma_D\colon
\lim_{i\to\infty}H_i(x)=-\infty\ \text{ and } \
\lim_{i\to-\infty}H_i(x)=\infty\right\}.\end{split}\]
Note that all the three sets are dense in $\Sigma_D$. The next lemma classifies elements of the set $M^{\rm e}(\Sigma_D)$ of shift-invariant ergodic measures on $\Sigma_D$.

\begin{lemma}
\label{trichotomy}
If $\mu\in M^{\rm e}(\Sigma_D)$, then either $\mu(A_0)=1$, 
$\mu(A_\alpha)=1$ or $\mu(A_\beta)=1$.
\end{lemma}
\begin{proof}
For the Dyck shift, the statement had been proved in 
\cite[pp.102--103]{Kri74}.
We treat the Dyck-Motzkin shift with a simpler argument.
%Let $D$ denote the set of all $x\in \Sigma_D$ such that there is $I=I(x)\in\mathbb Z$ with $H_I(x)>H_I(x)$ \textcolor{blue}{($H_I(x)>H_{I-p}$?)} and $H_I(x)\geq H_{I+p}(x)$, $p\in\mathbb N$. We claim that every $\mu\in M(\Sigma_D,\sigma)$ assigns measure $0$ to $D$. Indeed, we have $\sigma^i\{x\in D\colon I(x)=0\}=\{x\in D\colon I(x)=i\}$, $i\in\mathbb Z$, and if we had $\mu(D)>0$, then we would have the contradiction that the disjoint sets $\{x\in D\colon I(x)=i\}$ had equal and positive measure. In the same way we see that every $\mu\in M(\Sigma_D,\sigma)$ assigns measure $0$ to the set of $x\in \Sigma_D$ for which there is $I\in\mathbb Z$ such that one of the following conditions is satisfied: \[\begin{split}&H_I(x)\geq H_{I-p}(x)\ \text{ and } \ H_I(x)>H_{I+p}(x),\ p\in\mathbb N,\\ &H_I(x)< H_{I-p}(x)\ \text{ and } \ H_I(x)\leq H_{I+p}(x),\ p\in\mathbb N,\\ &H_I(x)\leq H_{I-p}(x)\ \text{ and } \ H_I(x)<H_{I+p}(x),\ p\in\mathbb N.\end{split}\] Hence the assertion of the lemma holds.
%(Another proof based on the proof appeared in Climenhaga's math blog).
With the notation in Section~\ref{s-w-l}, for $x\in\Sigma_D$ let
%$l(x):=\sup\{i\in\mathbb{Z}: x_i\text{ is negative and open (has no positive pair)}\}$
\[l(x)=\sup\left\{i\in\mathbb{Z}\colon
x\in\bigcup_{k=1}^M\left(\Sigma_{D}(i;\alpha_k)\cap \bigcup_{j=1}^\infty\{H_{i+j}=H_i\}^c\right)\right\},\]
and %$r(x):=\sup\{i\in\mathbb{Z}:x_i\text{ is positive and open}\}$.
\[r(x)=\inf\left\{i\in\mathbb{Z}\colon
x\in\bigcup_{k=1}^M\left(\Sigma_{D}(i;\beta_k)\cap \bigcup_{j=1}^\infty\{H_{i-j+1}=H_i\}^c\right)\right\},\]
where the upper indices $c$ denote the complement in $\Sigma_D$.
For $i,j\in\mathbb Z$ let
\[A_{i,j}=\{x\in\Sigma_D\colon l(x)=i,\ r(x)=j\}.\]
If $x\in \Sigma_D\setminus (A_0\cup A_{\alpha}\cup A_{\beta})$, then both $l(x)$ and
$r(x)$ are finite. Hence we have
\[\Sigma_D\setminus (A_0\cup A_{\alpha}\cup A_{\beta})\subset\bigcup_{i,j\in\mathbb{Z}}A_{i,j}.\]
For
all $i,j\in\mathbb{Z}$ we have $\sigma A_{i,j}=A_{i-1,j-1}$, and so
$A_{i,j}=A_{0,j-i}\neq\emptyset$.
%$\mu(A_{i,j})=\mu(A_{0,j-i})$ for any $i,j\in\mathbb{Z}$.
If $\mu\in M(\Sigma_D)$ then we have $\mu(A_{i,j})=0$. Since
$A_0$, $A_\alpha$, $A_\beta$ are shift-invariant, if $\mu$ is ergodic then
either $\mu(A_0)=1$, 
$\mu(A_\alpha)=1$ or $\mu(A_\beta)=1$.
\end{proof}

%The next lemma describes a relationship between $\{H_i\}_{i\in\mathbb Z}$ and relative frequencies of positive/negative symbols, and an analogous relationship for $\{H_{\gamma,i}\}_{i\in\mathbb Z}$. \begin{lemma}\label{diff-H}Let $j,k\in\mathbb Z$ satisfy $j<k$. \ \begin{itemize}    \item[(a)] Let $x\in\Sigma_D$. Then $H_j(x)=H_k(x)$ if and only if $k-j-@$ is even and \[{\rm card}\{i\in\{j,\ldots,k-1\}\colon x_i\in\{\alpha_1,\ldots,\alpha_M\}\}=\frac{1}{2}(k-j-@).\] \item[(b)] Let $y\in\Sigma_\alpha$. Then $H_{\alpha,j}(y)=H_{\alpha,k}(y)$ if and only if $k-j-@$ is even and \[{\rm card}\{i\in\{j,\ldots,k-1\}\colon y_i=\beta\}=\frac{1}{2}(k-j-@).\] \item[(c)] Let $y\in\Sigma_\beta$. Then $H_{\beta,j}(y)=H_{\beta,k}(y)$ if and only if $k-j-@$ is even and \[{\rm card}\{i\in\{j,\ldots,k-1\}\colon y_i=\alpha \}=\frac{1}{2}(k-j-@).\]\end{itemize} where $@$ denotes the number of neutral symbols in $\omega_j\cdots\omega_{k-1}$.\end{lemma}\begin{proof} Apply the definitions to each of the cases $0\leq j<k$, $j<k\leq0$, $j<0\leq k$ and compute.\end{proof}

\subsection{Construction of embeddings of the full shift}\label{DM-str}
We introduce two full shift spaces on $M+N+1$ symbols over different sub-alphabets of $D$:
  \[\Sigma_\alpha=(D_\alpha\cup D_0\cup\{\beta\})^{\mathbb Z}\quad\text{and}\quad\Sigma_\beta=(\{\alpha\}\cup D_0 \cup D_\beta)^{\mathbb Z}.\]
   Let $\sigma_\alpha$, $\sigma_\beta$ denote the left shifts acting on $\Sigma_\alpha$, $\Sigma_\beta$ respectively.
% The collection  $\{\Sigma(-j;\omega_1\cdots\omega_{2j+1})\colon j\geq0  \omega_1\cdots\omega_{2j+1}\in \mathcal L(\Sigma)\setminus\{\emptyset\}\}$  is a base of the topology on $\Sigma$. 
%With this notation,
With the notation in Section~\ref{s-w-l} we introduce two shift-invariant Borel sets of $\Sigma_D$:
\[\begin{split}B_\alpha&=\bigcap_{i=-\infty}^\infty\bigcup_{k=1}^M\bigcup_{\ell=1}^N\left(\Sigma_{D}(i;\alpha_k)\cup\Sigma_{D}(i;1_\ell)\cup\left(\Sigma_{D}(i;\beta_k)\cap\bigcup_{j=1}^\infty\{ H_{i-j+1}=H_{i+1}\}\right)\right),\\
B_\beta&=\bigcap_{i=-\infty}^\infty\bigcup_{k=1}^M\bigcup_{\ell=1}^N\left(\Sigma_{D}(i;\beta_k)\cup\Sigma_{D}(i;1_k)
\cup\left(\Sigma_{D}(i;\alpha_k)\cap \bigcup_{j=1}^\infty\{H_{i+j}=H_i\}\right)\right).\end{split}\]
%Note that $B_\alpha\supset\{\alpha_1,\ldots,\alpha_M,\varsigma_1,\ldots,\varsigma_N\}^\mathbb Z$ and $B_\beta\supset\{\beta_1,\ldots,\beta_M,\varsigma_1,\ldots,\varsigma_N\}^\mathbb Z$.
 The set $B_\alpha$ (resp. $B_\beta$)
is precisely the set of sequences in $\Sigma_D$ such that any right (resp. left) bracket in the sequence is closed.
One can check that
\[A_0\cup A_\alpha\subset B_\alpha\ \text{ and }\ A_0\cup A_\beta\subset B_\beta.\]

Define %\footnote{In \cite{Kri74}, this map for the Dyck shift was defined only on $B_\alpha$. }  %$\phi_\alpha\colon B_\alpha\to \Sigma_\alpha$
 $\phi_\alpha\colon \Sigma_D\to \Sigma_\alpha$ 
by
\[(\phi_\alpha(x))_i=\begin{cases}
  \beta&\text{ if }x_i\in  D_\beta,\\
 x_i &\text{ otherwise.}
\end{cases}\]
%\[(\phi_\alpha(x))_i=\begin{cases} \alpha_k&\text{ if }x_i=\alpha_k,\ k\in\{1,\ldots,M\},\\  \beta&\text{ if }x_i\in\{\beta_1,\ldots,\beta_M\},\\  1_\ell&\text{ if }x_i=1_\ell,\ \ell\in\{1,\ldots,N\}.\end{cases}\]
In other words, $\phi_\alpha(x)$ is obtained by replacing all $\beta_k$, $k\in\{1,\ldots,M\}$ in $x$ by $\beta$.
Clearly $\phi_\alpha$ is continuous, not one-to-one.
Similarly, define %\footnote{In \cite{Kri74}, this map for the Dyck shift was defined only on $B_\beta$. } 
$\phi_\beta\colon \Sigma_D\to \Sigma_\beta$ by
\[(\phi_\beta(x))_i=\begin{cases}
    \alpha&\text{ if }x_i\in D_\alpha,\\
    x_i&\text{ otherwise.}
\end{cases}\]
%\[(\phi_\beta(x))_i=\begin{cases}   \alpha&\text{ if }x_i\in\{\alpha_1,\ldots,\alpha_M\},\\    \beta_j&\text{ if }x_i=\beta_j,\ j\in\{1,\ldots,M\},\\  1_\ell&\text{ if }x_i=1_\ell,\ \ell\in\{1,\ldots,N\}.\end{cases}\]
In other words, $\phi_\beta(x)$ is obtained by replacing all $\alpha_k$, $k\in\{1,\ldots,M\}$ in $x$ by $\alpha$.
 Clearly $\phi_\beta$ is continuous, not one-to-one. 
 We set
\[K_\alpha=\phi_\alpha(B_\alpha)\ \text{ and }\ K_\beta=\phi_\beta(B_\beta).\]

%We will show that $\phi_\alpha|_{B_\alpha}$ and $\phi_\beta|_{B_\beta}$ are homeomorphisms onto their images.
    %  The statements for the Dyck shift were essentially proved in \cite[Section~4]{Kri74}. We extend the proofs there to the Dyck-Motzkin shift.
    %To this end, 

For each $j\in\mathbb Z$ define $G_{\alpha,j}\colon \Sigma_\alpha\to\{-1,0,1\}$ by \[G_{\alpha,j}(x)=\sum_{k=1}^{M}(\delta_{{\alpha_k},x_j}-\delta_{\beta,x_j}).\]
%We have $G_{\alpha,j}(x)=1$ if $x_j\in D_\alpha$, $G_{\alpha,j}(x)=0$ if $x_j\in D_0$, $G_{\alpha,j}(x)=-1$ if $x_j=\beta$.
    For each $i\in\mathbb Z$ define $H_{\alpha,i}\colon \Sigma_D\to\mathbb Z$ by
      \[H_{\alpha,i}(x)=\begin{cases}\sum_{j=0}^{i-1} G_{\alpha,j}(x)&\text{ for }  i\geq1,\\-\sum_{j=i}^{-1} G_{\alpha,j}(x)&\text{ for } i\leq -1,\\
      0&\text{ for }i=0.\end{cases}\]
      The function $H_{\alpha,i}$ for $i\geq1$ (resp. $i\leq-1$) counts the difference between the number of symbols in $D_\alpha$ and that of $\beta$ appearing in $x_0\cdots x_{i-1}$ (resp. $x_i\cdots x_{-1}$).
We define $\psi_\alpha\colon K_\alpha\to D^\mathbb Z$ by
\[(\psi_\alpha(y))_i=\begin{cases}
  \beta_k&\text{ if }y_i=\beta,\ y_{s_\alpha(i,y)}=
  \alpha_k,\ k\in\{1,\ldots,M\},\\
   y_i&\text{ otherwise,}
\end{cases}\]
%\[(\psi_\alpha(y))_i=\begin{cases} \alpha_k&\text{ if }y_i=\alpha_k,\ k\in\{1,\ldots,M\},\\ \beta_k&\text{ if }y_i=\beta,\ y_{s_\alpha(i,y)}=  \alpha_k,\ k\in\{1,\ldots,M\},\\   1_\ell&\text{ if }y_i=1_\ell,\ \ell\in\{1,\ldots,N\},\end{cases}\]
where \[s_\alpha(i,y)=\max\{j<i+1\colon  H_{\alpha,j}(y)= H_{\alpha,i+1}(y)\}.\]
Clearly $\psi_\alpha$ is continuous.

Similarly,  for each $j\in\mathbb Z$ define $G_{\beta,j}\colon \Sigma_\beta\to\{-1,0,1\}$ by \[G_{\beta,j}(x)=\sum_{k=1}^{M}(\delta_{{\alpha},x_j}-\delta_{\beta_k,x_j}).\]
    For each $i\in\mathbb Z$ define $H_{\beta,i}\colon \Sigma_D\to\mathbb Z$ by
      \[H_{\beta,i}(x)=\begin{cases}\sum_{j=0}^{i-1} G_{\beta,j}(x)&\text{ for }  i\geq1,\\-\sum_{j=i}^{-1} G_{\beta,j}(x)&\text{ for } i\leq -1,\\
      0&\text{ for }i=0.\end{cases}\]
We define $\psi_\beta\colon K_\beta\to D^\mathbb Z$ by 
\[(\psi_\beta(y))_i=\begin{cases}
   \alpha_k&\text{ if }y_i=\alpha,\ y_{s_\beta(i,y)}=\beta_k,\ k\in\{1,\ldots,M\},\\
    y_i&\text{ otherwise, }
\end{cases}\]
where \[s_\beta(i,y)=\min\{j>i\colon  H_{\beta,j}(y)= H_{\beta,i}(y)\}.\]
Clearly $\psi_\beta$ is continuous too.

\begin{lemma}\label{include-lem}
For each $\gamma\in\{\alpha,\beta\}$
%there exists a map $\psi_\gamma\colon K_\gamma\to B_\gamma$ such that 
the following statements hold: 
\begin{itemize}
\item[(a)] 
$\psi_\gamma(K_\gamma)=B_\gamma$, and $\psi_\gamma$ is a homeomorphism whose inverse is $\phi_\gamma|_{B_\gamma}$.
% $\phi_\gamma|_{B_\gamma}\colon B_\gamma\to K_\gamma$ is a homeomorphism whose inverse is $\psi_\gamma$.
\item[(b)] 
$\phi_\gamma\circ\sigma|_{B_\gamma}=\sigma_\gamma\circ\phi_\gamma|_{B_\gamma}$ and $\sigma^{-1}\circ\psi_\gamma=\psi_\gamma\circ\sigma_\gamma^{-1}|_{K_\gamma }$.
\end{itemize}
\end{lemma}
\begin{proof}  
For each $\gamma\in\{\alpha,\beta\}$,
it is straightforward to check that 
$\psi_\gamma\circ\phi_\gamma(x)=x$ for all $x\in B_\gamma$, and
$\phi_\gamma\circ\psi_\gamma(y)=y$
for all $y\in K_\gamma$, which verifies (a). A proof of (b) is also straightforward.\end{proof}

\subsection{Transport of ergodic measures}\label{transport-esc}By Lemma~\ref{include-lem},
if $\gamma\in\{\alpha,\beta\}$ and $\nu\in M(\Sigma_\gamma)$ satisfies $\nu(K_\gamma)=1$, then $\nu$ can be transported to a shift-invariant measure on $\Sigma_D$.
The lemma below gives a sufficient condition for measures in $M^{\rm e}(\Sigma_\gamma)$ to have measure $1$ on $K_\gamma$.
For each $\gamma\in\{\alpha,\beta\}$ we set
\[\Sigma_\gamma(\gamma)=\{\omega\in\Sigma_\gamma\colon \omega_0\in D_\gamma\} \ \text{ and }\ \Sigma_\gamma(0)=\{\omega\in\Sigma_\gamma\colon \omega_0\in D_0\},\]
and define $E_\gamma\colon\Sigma_\gamma\to\mathbb R$ by
\[E_\gamma=2\cdot \mathbbm{1}_{\Sigma_\gamma(\gamma) }+\mathbbm{1}_{\Sigma_\gamma(0)},\]
where $\mathbbm{1}_{(\cdot)}$ denotes the indicator function. Note that $E_\gamma$ is continuous. If $N=0$ then
$E_\gamma=2\cdot \mathbbm{1}_{\Sigma_\gamma(\gamma) }$.
If $\nu\in M(\Sigma_\alpha)$ and 
    $\int E_\alpha{\rm d}\nu>1$
    (resp. $\nu\in M(\Sigma_\beta)$ and 
    $\int E_\beta{\rm d}\nu>1$),
 then $\nu$ gives more mass to the union of cylinders corresponding to the symbols in $D_\alpha$ (resp. $D_\beta$) than the cylinder corresponding to $\beta$ (resp. $\alpha$).

\begin{lemma}\label{include-lem2} For each $\gamma\in\{\alpha,\beta\}$ the following statements hold: \begin{itemize}
    \item[(a)] If $\nu\in M^{\rm e}(\Sigma_\gamma)$ and 
    $\int E_\gamma{\rm d}\nu>1$
    then $\nu(K_\gamma )=1$.
    \item[(b)] $K_\gamma$ is a dense subset of $\Sigma_\gamma$.\end{itemize}\end{lemma}
\begin{proof}
The statements for the Dyck shift were essentially proved in \cite[Section~4]{Kri74}. We extend the proofs there to the Dyck-Motzkin shift.
By the definitions of $B_\alpha$ and $\phi_\alpha$ 
we have
\[\begin{split}K_\alpha&=\\
\bigcap_{i=-\infty}^\infty&\bigcup_{k=1}^M\bigcup_{\ell=1}^N\left(\Sigma_{\alpha}(i;\alpha_k)\cup \Sigma_{\alpha}(i;1_\ell)\cup\left(\Sigma_{\alpha}(i;\beta)\cap\bigcup_{j=1}^\infty\{ H_{\alpha,i-j+1}=H_{\alpha,i+1}\}\right)\right),\end{split}\]
and by De Morgan's laws,
\[\begin{split}K_\alpha^c&=\\
\bigcup_{i=-\infty}^\infty&\bigcap_{k=1}^M\bigcap_{\ell=1}^N\left(\Sigma_{\alpha}(i;\alpha_k)^c\cap\Sigma_{\alpha}(i;1_\ell)^c\cap\left(\Sigma_{\alpha}(i;\beta)^c\cup\bigcap_{j=1}^\infty\{ H_{\alpha,i-j+1}= H_{\alpha,i+1}\}^c\right)\right),\end{split}\]
where the upper indices $c$ denote the complements in $\Sigma_\alpha$. For each $y\in 
K_\alpha^c$ 
there exists $i\in\mathbb Z$ such that
$y_i=\beta$ and
$ H_{\alpha,i-j+1}(y)\neq H_{\alpha,i+1}(y)$ for all $j\in\mathbb N$.
%Then we have $\zeta_{i-1}=\beta$, for otherwise  $ H_{\alpha,i-1}(\zeta)=H_{\alpha,i+1}(\zeta)$. If $\zeta_{i-2}\neq\beta$ then we have $\zeta_{i-3}=\beta$, for otherwise $H_{\alpha,i-3}(\zeta)= H_{\alpha,i+1}(\zeta)$. Iterating this argument gives
By induction, for all $j\geq1$ we have
\[\#\{m\in\{i-j+1,\ldots,i\}\colon y_m=\beta\}> \#\{m\in\{i-j+1,\ldots,i\}\colon y_m\in D_\alpha \}.\]
If $\nu\in M^{\rm e}(\Sigma_\alpha)$ and 
$\int E_\alpha{\rm d}\nu>1$, then $\nu(\Sigma_{\alpha}(0;\beta))<\nu(\Sigma_{\alpha}(\alpha))$, and
Birkhoff's ergodic theorem applied to 
$(\sigma_\alpha^{-1},\nu)$
yields $\nu(K_\alpha)=1$ as required in (a) for $\gamma=\alpha$. One can treat the case $\gamma=\beta$ analogously, using $\sigma_\beta$ instead of $\sigma_\alpha^{-1}$.

In order to prove (b),
 let $\nu_\gamma$ denote the Bernoulli measure on 
$\Sigma_{\gamma}$ with the uniform distribution over  $M+N+1$ symbols, which are of entropy $\log(M+N+1)$. 
Since
$\int E_\gamma{\rm d}\nu_\gamma>1$ by hypothesis, we have $\nu_\gamma(K_\gamma)=1$ by (a). Since $\nu_\gamma$ is fully supported on $\Sigma_\gamma$, $K_\gamma$ is a dense subset of $\Sigma_\gamma$. The proof of Lemma~\ref{include-lem2} is complete.
 \end{proof}

 \subsection{Borel embeddings}\label{Borel-sect}
Let $\Sigma_1$, $\Sigma_2$ be two subshifts.
For $i=1,2$,
let $\sigma_i$ denote the left shift acting on $\Sigma_i$.
Let $K$ be a shift-invariant, proper Borel subset of $\Sigma_1$. 
We say $\psi\colon K\to\Sigma_2$ is a {\it Borel embedding} of $\Sigma_1$ into $\Sigma_2$ if the following hold:
\begin{itemize}
\item[(i)] $\psi$ is continuous, injective
and $\psi^{-1}\colon \psi(K)\to K$ is continuous,
\item[(ii)] $\sigma_2\circ\psi=\psi\circ\sigma_1$,
%\item $\sup\{h(\mu\circ\psi^{-1})\colon\mu\in M(\Sigma_1)\text{ and }\mu(K_1)=1\}=h_{\rm top}(\Sigma_1).$
\item[(iii)] there is no continuous map $\bar\psi\colon\Sigma_1\to\Sigma_2$ such that $\bar\psi=\psi$ on $K$.
\end{itemize}

By
\cite[Theorem~5.3]{HI05} and \cite[Theorem~5.8]{HIK09}, there is no proper embedding of the full shift on $M+N+1$ symbols to $\Sigma_D$. From this and Lemma~\ref{include-lem} it follows that 
$K_\gamma$ is a proper subset of $\Sigma_\gamma$.
Moreover the following holds. 
\begin{prop}\label{emb-prop}For each $\gamma\in\{\alpha,\beta\}$ the map $\psi_\gamma\colon K_\gamma\to B_\gamma\subset\Sigma_D$ is a Borel embedding of $\Sigma_\gamma$ into $\Sigma_D$.\end{prop}
\begin{proof}Conditions (i), (ii) in the definition of Borel embedding is a consequence of 
Lemma~\ref{include-lem}.
It is left to show (iii).
If there were a continuous map $\bar\psi_\gamma\colon\Sigma_\gamma\to\Sigma_D$ such that $\bar\psi_\gamma(y)=\psi_\gamma(y)$ for all $y\in K_\gamma$,
then $\bar\psi_\gamma\circ\phi_\gamma\colon\Sigma_D\to\Sigma_D$ would be continuous. Since
$\psi_\gamma\circ\phi_\gamma(x)=x$ holds
for all $x\in B_\gamma$,
$\bar\psi_\gamma\circ\phi_\gamma (x)=x$ would hold
for all $x\in B_\gamma$.
Since $B_\gamma$ is dense in $\Sigma_D$
Lemma~\ref{include-lem2}(b), $\bar\psi_\gamma\circ\phi_\gamma (x)=x$
would hold for all $x\in \Sigma_D$. Then
$\phi_\gamma$ would be injective, a contradiction. \end{proof}
 
\subsection{Approximation of ergodic measures by $CO$-measures}\label{approx-sec}
Let $\Sigma$ be a subshift.
A point $x\in \Sigma$ is called a {\it  periodic point} of period $n\in\mathbb N$
if $\sigma^nx=x$.
An element of $M(\Sigma)$ that is supported on the orbit of a single periodic point is called a {\it $CO$-measure}.
Clearly $CO$-measures are ergodic.
Let $M^{CO}(\Sigma)$ denote the set of $CO$-measures.

Recall that
 there exist
three pairwise disjoint shift-invariant Borel
subsets $A_0$, $A_\alpha$, $A_\beta\subset\Sigma_D$ such that
any shift-invariant ergodic measure has measure $1$ to one of these three sets (Section~\ref{iso-sec} and Lemma~\ref{trichotomy}).
Let
\[M^{\rm e}_{\gamma}(\Sigma_D)=\{\mu\in M^{\rm e}(\Sigma_D)\colon \mu(A_\gamma)=1\}\ \text{ for } \gamma\in\{0,\alpha,\beta\}.\]

 \begin{prop}\label{per-approx}\
 For all $\gamma\in\{\alpha,\beta\}$
 we have
  \[M^{\rm e}_{\gamma}(\Sigma_D)\subset \overline{M^{\rm e}_{\gamma}(\Sigma_D)\cap M^{CO}(\Sigma_D)}\]
and \[M^{\rm e}_{0}(\Sigma_D)\subset \overline{M^{\rm e}_{\gamma}(\Sigma_D)\cap M^{CO}(\Sigma_D)}.\]
In particular, $M^{CO}(\Sigma_D)$ is dense in $M^{\rm e}(\Sigma_D)$.
\end{prop}
% As a corollary to Lemmas~\ref{per-approx} and \ref{per-approx3} we obtain the following statement. 
 %\begin{cor}\label{per-cor}For each $\gamma\in\{\alpha,\beta\}$ we have \[M^{\rm e}(\Sigma_D,\sigma)\subset \overline{M^{\rm e}_{\gamma}(\Sigma_D,\sigma)\cap M^{CO}(\Sigma_D,\sigma)}.\]\end{cor}

\begin{proof} For $\omega=\omega_1\cdots\omega_k\in D^*$ and $n\in\mathbb N$, 
let $\omega^n\in D^*$ denote the $n$-fold concatenation: $(\omega^n)_i=\omega_{j}$, $j=i$ mod $k$ for $i=1,\ldots,kn$. 
Let $\omega\in D^*$. 
%We say $\omega$ is a
 %{\it periodic defining word}
%\textcolor{red}{The def of periodic defining block in \cite[p.108]{HI05} wrong?}
%if $\omega^n$ is admissible for all $n\in\mathbb N$.
%If $\omega\in D^*$ is a periodic defining word of $\Sigma_D$, then $\Sigma_D(0;\omega)$ contains exactly one periodic point of period equal to the word length of $\omega$. 
%$|\omega|$.
 If ${\rm red}(\omega)=1$ then we say $\omega$ is 
{\it neutral}. 
%The empty word $\emptyset$ is said to be neutral.
 If ${\rm red}(\omega)$ is a concatenation of symbols in $D_\alpha$ (resp. $D_\beta$),
then we say $\omega$ is {\it negative}
(resp. {\it positive}). If $\omega$ is neutral, negative or positive then $\omega^n$ is admissible for all $n\in\mathbb N$. 

In order to prove the first inclusion, 
let $\mu\in M^{\rm e}_{\gamma}(\Sigma_D)$.
  By Birkhoff's ergodic theorem,
  there exists $x\in A_\gamma$ such that $n^{-1}\sum_{i=0}^{n-1}\delta_{\sigma^ix}\to\mu$ and $n^{-1}\sum_{i=0}^{-n+1}\delta_{\sigma^ix}\to\mu$,
where $\delta_{\sigma^ix}$ denotes the unit point mass at $\sigma^ix$.  
     If $\gamma=\alpha$, then
  by the definition of $A_{\alpha}$, 
for any $n\in \mathbb N$
  there exists $l\in\mathbb Z$ such that $l\le -n$ and $\omega=x_{l}\cdots x_n$ is negative. 
  Then $\Sigma_D(l;x_{l}\cdots x_n)$ contains exactly one periodic point of period $n-l+1$. Let $\mu_n$ denote the $CO$-measure supported on the orbit of this periodic point.
   Since $\{x\}=\bigcap_{n=1}^\infty\Sigma_D(l;x_{l}\cdots x_n)$ we have
     $\mu_n\in M^{\rm e}_{\alpha}(\Sigma_D)$ and $\mu_n\rightarrow \mu$.
  A proof for the case $\gamma=\beta$ is analogous.

In order to prove the second inclusion, let $\mu\in M^{\rm e}_{0}(\Sigma_D)$.
 By Birkhoff's ergodic theorem,
  there exists $x\in A_0$ such that $n^{-1}\sum_{i=0}^{n-1}\delta_{\sigma^ix}\to\mu$ and $n^{-1}\sum_{i=0}^{-n+1}\delta_{\sigma^ix}\to\mu$. 
By the definition of $A_0$, for any $n\in\mathbb N$ there exist $l$, $m\in\mathbb Z$ such that $l\le -n$, $m\ge n$ and $\omega=x_{l}\cdots x_{m}$ is neutral. For all $k\in\mathbb N$,
   $\omega^{2k}\alpha_1$ is admissible, negative and $\Sigma_D(-k|\omega|+l;\omega^{2k}\alpha_1)$
contains exactly one periodic point of period $2k|\omega|+1$. Let $\mu_n$ denote the $CO$-measure supported on the orbit of this periodic point.
Clearly we have $\mu_n\in M^{\rm e}_{\alpha}(\Sigma_D)$.
Since $\{x\}=\bigcap_{n=1}^\infty \Sigma_D(-k|\omega|+l;\omega^{2k})$, we obtain $\mu_n\to\mu$.
This verifies the second inclusion for $\gamma=\alpha$.
A proof for $\gamma=\beta$ is analogous, with all $\alpha_1$ replaced by $\beta_1$. 
%Let $\mu\in M^{\rm e}_{0}(\Sigma_D,\sigma)$ be a $CO$-measure supported on the orbit of a periodic point $x=(x_n)_{n}\in\Sigma_D$ of period $p$. Note that $p$ is even. Set $v=x_{-p/2}\cdots x_0\cdots x_{p/2}$.   Without loss of generality we may assume $x_{-p/2}$ is right and $x_{p/2}$ is left. We have  $x\in[v]_{p/2}$, and by Lemma~\ref{factorization}(a), there are a neutral or negative word $v(\alpha)$, and a neutral or positive word $v(\beta)$, %$w\in L(\Sigma_D)$  such that  $v=v(\beta)v(\alpha)$. 
  %${\rm red}(w)=1$ and  The condition $\mu\in M^{\rm e}_{0}(\Sigma_D,\sigma)$ implies $|{\rm red}(v(\alpha))|=|{\rm red}(v(\beta))|$. \[ w_\beta=\begin{cases}    \beta_1&\text{ if }v(\alpha)=\emptyset\\    \rho^*(v(\alpha))&\text{ otherwise.}\end{cases}\] We have $v^{2n}\rho^*(v(\alpha))\in L(\Sigma_D)$ and is a periodic defining block.  $[v^{2n}w_\beta]_{pn+p/2}$ contains a periodic point whose orbit supports a  closed orbit measure in $M^{\rm e}_{\beta}(\Sigma_D,\sigma)$. Since $\{x\}=\bigcap_{n\geq1} [v^{2n}w_\beta]_{pn+p/2}$,these closed orbit measures approximate $\mu$.
 \end{proof}

 \section{On the abundance of high complexity paths}
 The aim of this section is to prove the abundance of high complexity paths of ergodic measures for the Dyck-Motzkin shift informally stated in (1) (2) in Section~1.
 In Section~\ref{abundance} we give relevant definitions, and give a precise statement of this in
 Proposition~\ref{join-lem}.
 After proving two preliminary lemmas in Section~\ref{pre-lem}, we complete the proof of Proposition~\ref{join-lem} in Section~\ref{pfprop}.
 In Section~\ref{add-sec} we comment more on the path connectedness of spaces of ergodic measures.

\subsection{A precise statement}\label{abundance}

A path $t\in[0,1]\mapsto \mu_t\in M^{\rm e}(\Sigma_D)$ is called a {\it high complexity path} if the following two conditions hold:
\begin{itemize}
\item[(A1)]
$\mu_t$ is fully supported on $\Sigma_D$ and is Bernoulli
for all $t\in[0,1]$ but countably many values;
\item[(A2)]
For any $\mu\in \{\mu_t\colon t\in[0,1]\}$, the set
$\{t\in[0,1]\colon\mu_t=\mu\}$ is countable.
\end{itemize}
The next proposition is a key ingredient in the proof of Theorem~A(b) and that of Theorem~B.

\begin{prop}[The abundance of high complexity paths]
\label{join-lem}
Let $U$ be a convex open subset of $M(\Sigma_D)$, let $\gamma\in\{\alpha,\beta\}$ and let $\mu^+$, $\mu^-\in U\cap (M^{\rm e}_{0}(\Sigma_D)\cup M^{\rm e}_{\gamma}(\Sigma_D))$ be distinct measures.
There exists a high complexity path that lies in $U\cap (M^{\rm e}_{0}(\Sigma_D)\cup M^{\rm e}_{\gamma}(\Sigma_D))$ and joins $\mu^+$, $\mu^-$.
\end{prop}

From Proposition~\ref{join-lem} we immediately obtain  
 the following statement.
 %that will be used in the proof of Theorem~B.
\begin{prop}\label{cor-path-1}
The spaces $M^{\rm e}_{0}(\Sigma_D)\cup M^{\rm e}_{\alpha}(\Sigma_D)$ and $M^{\rm e}_{0}(\Sigma_D)\cup M^{\rm e}_{\beta}(\Sigma_D)$ are path connected and locally path connected.\end{prop}
\begin{proof}The path connectedness is a consequence of Proposition~\ref{join-lem} with $U=M(\Sigma_D)$.
Since any point in $M(\Sigma_D)$ has a neighborhood base consisting of convex open sets,
the local path connectedness is also a consequence of Proposition~\ref{join-lem}.
\end{proof}

A proof of Proposition~\ref{join-lem} has been inspired by the works of Sigmund \cite{Sig70,Sig77} on the path connectedness of the space of ergodic measures for
Axiom A diffeomorphisms (essentially topologically mixing Markov shifts).
In \cite{Sig70} he proved that the set of $CO$-measures are dense in the space of shift-invariant Borel probability measures. Later in the proof of \cite[Theorem~B]{Sig77} he showed  that any pair of $CO$-measures can be joined by a path of ergodic measures, in such a way that if the two $CO$-measures lie in a convex open set, then the whole path can be chosen to lie in this set.
%if one of the two CO-measures lies in a convex neighborhood of the other one, then the whole path can be chosen to lie in this neighborhood. 
These paths can be concatenated to form a path joining a given pair of ergodic measures. 
%In order to prove Proposition~\ref{join-lem} further develop Sigmund's argument for the Dyck-Motzkin shift.

The rest of this section except Section~\ref{add-sec} is dedicated to
a proof of Proposition~\ref{join-lem} that breaks into three steps.
We only give a proof for $\gamma=\alpha$ since that
for $\gamma=\beta$ is identical.
Using Proposition~\ref{per-approx},
we take sequences $\{\mu_n^+\}_{n\in\mathbb N}$, $\{\mu_{n}^-\}_{n\in\mathbb N}$ of $CO$-measures in $U$ that approximate $\mu^+$, $\mu^-$ respectively.
Then we transport 
$\{\mu_n^+\}_{n\in\mathbb N}$, $\{\mu_{n}^-\}_{n\in\mathbb N}$ %to $\Sigma_\alpha$
 via $\phi_\alpha|_{B_\alpha}\colon B_\alpha\to K_\alpha\subset\Sigma_\alpha$, and apply Sigmund's result in the proof of \cite[Theorem~B]{Sig77} (see Lemma~\ref{Sig-lem}) to
obtain sequences of paths that appropriately join the transported measures.  
Finally we transport these paths back to $\Sigma_D$ via $\psi_\alpha\colon K_\alpha\to B_\alpha\subset\Sigma_D$,
and concatenate all of them to form a path joining $\mu^+$, $\mu^-$ with the required property.
%For each $n\in\mathbb N$ we will then construct intermediate  paths joining $\mu_n^+$, $\mu_{n+1}^{+}$ and $\mu_n^-$, $\mu_{n+1}^{-}$, together with a path joining  $\mu_1^+$, $\mu_{1}^-$. Finally we will concatenate all these paths to form a path joining $\mu^+$, $\mu^-$ with the required property. All the intermediate paths will be obtained by transporting these paths.
The transport in the last step needs to be justified since $\psi_\alpha$ is a Borel embedding of $\Sigma_\alpha$ into $\Sigma_D$
as in Proposition~\ref{emb-prop}, which does not have a continuous extension to the whole shift space $\Sigma_\alpha$. 

\begin{remark}\label{GP-rem}
%Prior to our work, 
Gorodetski and Pesin \cite{GP17} further developed Sigmund's argument explained as above to prove the path connectedness of some basic pieces of the space of ergodic measures for some partially hyperbolic diffeomorphisms. In \cite{GP17}, they %left off 
did not treat the path connectedness of the whose space of ergodic measures.
\end{remark}

\subsection{Preliminary lemmas}\label{pre-lem}
Recall that $\sigma_\alpha K_\alpha=K_\alpha$ but
 $K_\alpha$ is not a subshift as it is not closed. With a slight abuse of notation, let $M(K_\alpha)$ denote the space of $\sigma_\alpha|_{K_\alpha}$-invariant Borel probability measures on $K_\alpha$ endowed with the weak* topology.
Note that 
 \[M(K_\alpha)=\{\nu\in M(\Sigma_\alpha)\colon
\nu(K_\alpha)=1\}.\]
By Lemma~\ref{include-lem}(a),
 $\phi_\alpha|_{B_\alpha}\colon B_\alpha\to K_\alpha$ is a homeomorphism whose inverse is the Borel embedding $\psi_\alpha\colon K_\alpha\to B_\alpha$.
 Define a push-forward $\psi_{\alpha}^*\colon M(K_\alpha) \to M(\Sigma_D)$ by
\[\psi_{\alpha}^*(\nu)=\nu\circ\psi_\alpha^{-1}.\]
Since $\psi_\alpha$ is continuous, $\psi_\alpha^*$ is continuous.
\begin{lemma}\label{support}
If $\nu\in M(K_\alpha)$
is fully supported on $\Sigma_\alpha$,
then $\psi_\alpha^*(\nu)$ is fully supported on $\Sigma_D$.
\end{lemma}
\begin{proof}

Since $\psi_\alpha$ is a homeomorphism by Lemma~\ref{include-lem}(a) and $K_\alpha$ is dense in $\Sigma_\alpha$
by Lemma~\ref{include-lem2}(b), if 
$\nu\in M(K_\alpha)$
is fully supported on $\Sigma_\alpha$ then
 $\nu\circ\psi_\alpha^{-1}(B_\alpha)=\nu(K_\alpha)=1$.  Since $B_\alpha$ is dense in $\Sigma_D$, $\psi_\alpha^*(\nu)$ is fully supported on $\Sigma_D$. 
\end{proof}

The next lemma is essentially due to Sigmund, shown in the proof of \cite[Theorem~B]{Sig77}.
Here we only add a supplementary proof.
\begin{lemma}\label{Sig-lem}
Let $\Sigma$ be a topologically mixing Markov shift. Let $\mu$, $\mu'\in M^{\rm e}(\Sigma)$ be distinct CO-measures, and
let $U$ be a convex open subset of $M^{\rm e}(\Sigma)$ that contains $\mu$, $\mu'$.
There is a homeomorphism
 $t\in[0,1]\mapsto\theta_t\in U$ onto its image
that is a high complexity path %lying in $U$ and 
joining $\mu$, $\mu'$.
\end{lemma}
\begin{proof} It was shown in the proof of \cite[Theorem~B]{Sig77} that there exist a topologically mixing Markov shift $\Sigma'$, a homeomorphism $\tau\colon\Sigma\to\Sigma'$ commuting with the shifts, and a path $t\in[0,1]\mapsto \theta_t\in U$ joining $\mu$, $\mu'$ such that $\theta_t\circ \tau^{-1}$ is a Markov measure for all $t\in(0,1)$.
In particular,  $\theta_t\circ\tau^{-1}$ is a Gibbs state \cite{Bow75} for all $t\in(0,1)$, and it is Bernoulli by \cite[Theorem~1.25]{Bow75} and \cite{FO70}.
Hence, $\theta_t$ is fully supported on $\Sigma$ and is  Bernoulli for all $t\in(0,1)$.
A close inspection into the proof of \cite[Theorem~B]{Sig77} shows the injectivity of $t\in[0,1]\mapsto\theta_t\in U$.
\end{proof}

\subsection{Proof of Proposition~\ref{join-lem}}\label{pfprop} 
We fix a countable dense subset $\{f_{n}\not\equiv0\colon n\in\mathbb N\}$ of $C(\Sigma_\alpha)$, and define 
a metric $d$ on $M(\Sigma_\alpha)$ by
\[d(\nu,\nu')=\sum_{n=1}^\infty\frac{|\int f_{n}{\rm d}\nu-\int f_{n}{\rm d}\nu'|}{2^n\|f_{n}\|_{C^0}}\ \text{ for } \nu, \nu'\in M(\Sigma_\alpha).\]
The weak* topology on
$M(\Sigma_\alpha)$ is metrizable by the metric $d$.
For $\nu\in M(\Sigma_\alpha)$ and $\delta>0$,
let $U_\alpha(\nu,\delta)$ denote the open ball of radius $\delta$ about $\nu$ with respect to $d$.
%Note that  $U_\alpha(\nu,\delta)$ is a convex set. 

Let $U$ be a convex open subset of $M(\Sigma_D)$ and let $\mu^+$, $\mu^-\in U\cap (M^{\rm e}_{0}(\Sigma_D)\cup M^{\rm e}_{\alpha}(\Sigma_D))$ be distinct measures.
The definition of $E_\alpha$ gives
\begin{equation}\label{Delta-eq}\int E_\alpha\circ\phi_\alpha {\rm d}\mu^+\geq1\ \text{ and }\
\int E_\alpha\circ\phi_\alpha {\rm d}\mu^-\geq1.\end{equation}
Since $U$ is open and $\psi_\alpha^*$ is continuous, there exists $q\in\mathbb N$ such that
 \begin{equation}\label{delta-d}\psi_\alpha^*\left(\left(U_\alpha\left(\mu^+\circ\phi_\alpha^{-1},q^{-1}\right)\cup U_\alpha\left(\mu^-\circ\phi_\alpha^{-1},q^{-1} \right)\right)\cap M(K_\alpha)\right)\subset U.\end{equation}
% \textcolor{green}{replace: $\mu\circ\phi_\alpha^{-1}\to\mu\circ\psi_\alpha$}
By \eqref{Delta-eq} and
 Proposition~\ref{per-approx}(a)(b), there exist sequences $\{\mu_n^+\}_{n\in\mathbb N}$,
$\{\mu_{n}^-\}_{n\in\mathbb N}$ of 
$CO$-measures in $U\cap M^{\rm e}_{\alpha}(\Sigma_D)$ converging to $\mu^+$,
$\mu^-$ respectively such that $\mu_1^+\neq\mu_1^-$, $\mu_n^+\neq\mu_{n+1}^+$ and $\mu_n^-\neq\mu_{n+1}^-$ for all $n\in\mathbb N$, and the following two conditions hold:

\begin{itemize}
\item[(B1)] For all $n\in\mathbb N$, 
$\int E_\alpha\circ\phi_\alpha{\rm d}\mu_n^+>1$ and $\int E_\alpha\circ\phi_\alpha{\rm d}\mu_{n}^->1$.

\item[(B2)] For all $n\in\mathbb N$,
 \[\max\{d(\mu^+\circ\phi_\alpha^{-1},\mu_n^+\circ\phi_\alpha^{-1}),d(\mu^-\circ\phi_\alpha^{-1},\mu_{n}^-\circ\phi_\alpha^{-1})\}<(n+q)^{-1}.\]
\end{itemize}

Since $E_\alpha$ is continuous, the set
\[V_\alpha=\left\{\nu\in M(\Sigma_\alpha)\colon \int E_\alpha{\rm d}\nu>1\right\}\]
 is a convex open subset of $M(\Sigma_\alpha)$.
For each $n\in\mathbb N$ we set 
\[\begin{split}V_n^+=U_\alpha\left(\mu^+\circ\phi_\alpha^{-1},(n+q)^{-1}\right)\cap V_\alpha\ \text{ and }\ W_n^+=\psi_\alpha^*(V_n^+\cap M(K_\alpha)),\\
V_n^-=U_\alpha\left(\mu^-\circ\phi_\alpha^{-1},(n+q)^{-1}\right)\cap V_\alpha\ \text{ and }\ W_n^-=\psi_\alpha^*(V_n^-\cap M(K_\alpha)).\end{split}\]
% and
% \[\begin{split}W_n^+&=\left\{\mu\in M(\Sigma_D)\colon \mu\circ\phi_\alpha^{-1}\in V_n^+\right\},\\  W_n^-&=\left\{\mu\in M(\Sigma_D)\colon  \mu\circ\phi_\alpha^{-1}\in V_n^-\right\}.\end{split}\]
Note that $V_n^+$, $V_n^-$ are convex open subsets of $M(\Sigma_\alpha)$, and they decrease as $n$ increases. By (B1) and (B2), $\mu_n^+\circ\phi_\alpha^{-1}$ belongs to $V_n^+\cap M(K_\alpha)$ and $\mu_n^-\circ\phi_\alpha^{-1}$ belongs to $V_n^-\cap M(K_\alpha)$.
%Since $\mu^+\circ\phi_\alpha^{-1}$ and $\mu^-\circ\phi_\alpha^{-1}$ are contained in the closure of $V_\alpha$ by \eqref{Delta-eq}, both $V_n^+$ and $V_n^-$ are non-empty.
%and decreasing sequences of convex open sets.
% it follows that $\{W_n^+\}_{n\in\mathbb N}$, $\{W_{n}^-\}_{n\in\mathbb N}$ are non-empty, decreasing sequences of open sets.
  By \eqref{delta-d} we have 
 \begin{equation}\label{U}W_n^+\cup W_n^-\subset U\ \text{ for all }n\in\mathbb N.\end{equation}
\begin{figure}
\begin{center}
\includegraphics[height=7.5cm,width=7.5cm]
{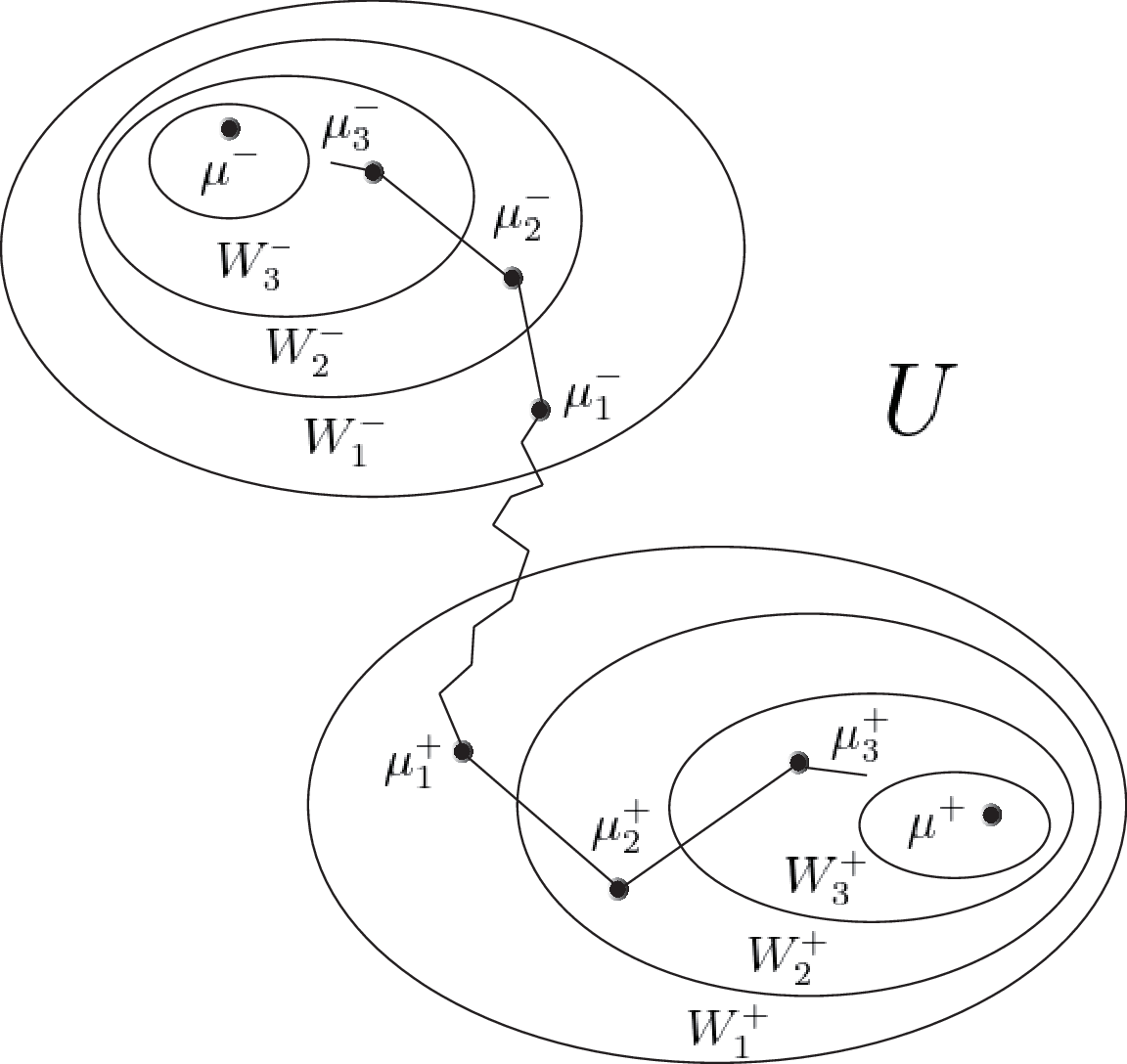}
\caption
{Each straight segment indicates the path that lies in $U\cap M_\alpha^{\rm e}(\Sigma_D)$ and joins the two $CO$-measures at its endpoints, which is obtained by applying Lemma~\ref{Sig-lem} once. Their concatenation lies in $U\cap M_\alpha^{\rm e}(\Sigma_D)$ and joins $\mu^+$, $\mu^-$.}\label{mu}
\end{center}
\end{figure}
If $\int E_\alpha\circ\phi_\alpha {\rm d}\mu^+>1$
(resp. $\int E_\alpha\circ\phi_\alpha {\rm d}\mu^+=1$), then $\mu^+\in W_n^+$ (resp. $\mu^+\notin W_n^+$) holds
for all $n\in\mathbb N$. 
Analogous statements hold for $\mu^-$. 
 The next lemma asserts that $W_n^+$ (resp. $W_{n}^-$) approaches to $\mu^+$ (resp. $\mu^-$) as $n\to\infty$.
See \textsc{Figure}~\ref{mu} for a schematic picture.

\begin{lemma}\label{contained-lem}
For any open subset $Y$ of $M(\Sigma_D)$ that contains $\mu^+$ (resp. $\mu^-$), there exists $n\in\mathbb N$ such that
$M^{\rm e}(\Sigma_D)\cap W_{n}^+\subset Y$ (resp. $M^{\rm e}(\Sigma_D)\cap W_{n}^-\subset Y$).
\end{lemma}
\begin{proof}
 Let
$\{\rho_{n}\}_{n\in\mathbb N}$ be a sequence in $M^{\rm e}(\Sigma_D)$ such that
$\rho_n\in M^{\rm e}(\Sigma_D)\cap W_{n}^+$ for all $n\in\mathbb N$.
Let $\{\rho_{n(j)}\}_{j\in\mathbb N}$ be an arbitrary convergent subsequence and 
let $\rho$ denote its limit measure.
By the definitions of $V_{n(j)}^+$ and $W_{n(j)}^+$ we have $d(\mu^+\circ\phi_\alpha^{-1},
\rho_{n(j)}\circ\phi_\alpha^{-1})<(n(j)+q)^{-1}$, and so $\rho_{n(j)}\circ\phi_\alpha^{-1}\to\mu^+\circ\phi_\alpha^{-1}$.
Meanwhile we have
$\psi_\alpha^*(\rho_{n(j)}\circ\phi_\alpha^{-1})=\rho_{n(j)}\to\rho$ 
and 
$\psi_\alpha^*(\mu_{n(j)}^+\circ\phi_\alpha^{-1})=\mu_{n(j)}^+\to\mu^+$, and
$\mu_{n(j)}^+\circ\phi_\alpha^{-1}\to\mu^+\circ\phi_\alpha^{-1}$ by (B2).
The continuity of $\psi_\alpha^*$ at $\mu^+\circ\phi_\alpha^{-1}$ yields $\rho=\mu^+$. Since 
$\{\rho_{n(j)}\}_{j\in\mathbb N}$ is an arbitrary convergent subsequence, the assertion of the lemma for $\mu^+$ follows.  
A proof of the assertion of the lemma for $\mu^-$ is analogous.
\end{proof}

%\begin{proof} Suppose by contradiction there exists an open set $U$ that contains $\mu^+$, and a sequence $\{\rho_{n}\}_{n\in\mathbb N}$ such that $\rho_n\in M^{\rm e}(\Sigma_D)\cap W_{n}^+-U$.  We have $\rho_{n}\circ\phi_\alpha^{-1}\to\mu^+\circ\phi_\alpha^{-1}$ as $n\to\infty$ in the weak* topology on $M(\Sigma_\alpha)$. Pick an weak* accumulation point of $\{\rho_{n}\}_{n\in\mathbb N}$ and denote it by $\rho$. Since $\mu^+$ is contained in the open set $U$, we have $\rho\neq\mu^+$.  Meanwhile, we have $\psi_{\alpha,*}(\rho_{n}\circ\phi_\alpha^{-1})=\rho_{n}\to\rho$,   and $\psi_{\alpha,*}(\mu_n^+\circ\phi_\alpha^{-1})=\mu_n^+\to\mu^+$ as $n\to\infty$ by (B2) and the continuity of $\psi_{\alpha,*}$. This yields a contradiction to the continuity of $\psi_{\alpha,*}$ at $\mu^+\circ\phi_\alpha^{-1}$. A proof of the statement of the lemma for $\mu^-$ is analogous. \end{proof}

By Lemma~\ref{Sig-lem}, for each $n\in\mathbb N$ there is a homeomorphism 
$t\in [0,1]\mapsto \theta_{n,t}^+\in V_n^+\cap M^{\rm e}(\Sigma_\alpha)$ onto its image such that $\theta_{n,0}^+=\mu_n^+\circ\phi_\alpha^{-1}$, $\theta_{n,1}^+=\mu_{n+1}^+\circ\phi_\alpha^{-1}$, and 
$\theta_{n,t}^+$ is fully supported on $\Sigma_\alpha$ and is Bernoulli 
for all $t\in(0,1)$.
Since $V_n^+\subset V_\alpha$, we have $\int E_\alpha{\rm d} \theta_{n,t}^+>1$
for all $t\in[0,1]$. Lemma~\ref{include-lem2}(a) gives
 $\theta_{n,t}^+\in M(K_\alpha)$ for all $t\in[0,1]$.
Hence, the measure $\mu_{n,t}^+=\theta_{n,t}^+\circ\psi_\alpha^{-1}$ is well-defined and belongs to $M^{\rm e}_{\alpha}(\Sigma_D)$ for all $t\in[0,1]$, and
 $t\in[0,1]\mapsto \mu_{n,t}^+\in M^{\rm e}_{\alpha}(\Sigma_D)$ is a homeomorphism onto its image.
This path
lies in $W_n^+$ and joins $\mu_n^+$, $\mu_{n+1}^+$ by 
 $ \mu_{n,0}^+=\mu_n^+$ and $ \mu_{n,1}^+=\mu_{n+1}^+$. By Lemma~\ref{support}, $\mu_{n,t}^+$ is fully supported on $\Sigma_D$ and is Bernoulli
 for all $t\in(0,1)$.
 In other words,
$t\in[0,1]\mapsto\mu_{n,t}^+$ is a high complexity path. We repeat the same argument to obtain
for each $n\in\mathbb N$
a high complexity path
$t\in[0,1]\mapsto \mu_{n,t}^-\in M^{\rm e}_{\alpha}(\Sigma_D)$ that lies in $W_n^-$ and joins $\mu_n^-$, $\mu_{n+1}^-$ by 
 $ \mu_{n,0}^-=\mu_n^-$ and $ \mu_{n,1}^-=\mu_{n+1}^-$.

%Since $V$ is convex, Lemma~\ref{Sig-lem} ensures the existence of a path that lies in $V$ and joins $\mu_1^+\circ\phi_\alpha^{-1}$, $\mu_1^-\circ\phi_\alpha^{-1}$. However, the pullback of this path may not lie in $U_0$.
%To construct a path joining the remaining $\mu_1^+$ and $\mu_1^-$, a different argument is needed. 
 \begin{lemma}\label{1-path}
There is a high complexity path that lies in $U\cap M_{\alpha}^{\rm e}(\Sigma_D)$ and joins $\mu_1^+$, $\mu_{1}^-$.
\end{lemma}
\begin{proof}
 Consider the set
 $L=
\{t\mu_1^+\circ\phi_\alpha^{-1}+(1-t)\mu_{1}^-\circ\phi_\alpha^{-1}\colon t\in[0,1]\}.$
By (B1) and Lemma~\ref{include-lem2}(a) we have
$\mu_1^+\circ\phi_\alpha^{-1}$, $\mu_1^-\circ\phi_\alpha^{-1}\in M(K_\alpha)$. Hence
 $L\subset M(K_\alpha)$ holds.
Since $\psi_\alpha^*(\mu_1^+\circ\phi_\alpha^{-1})=\mu_1^+\in U$, $\psi_\alpha^*(\mu_1^-\circ\phi_\alpha^{-1})=\mu_1^-\in U$ and $U$ is convex, we have
 $\psi_\alpha^*(L)=\{t\mu_1^++(1-t)\mu_{1}^-\colon t\in[0,1]\}\subset U.$
Since $\psi_\alpha^*$ is continuous, $L\subset V_\alpha$ and
 $L$ is a compact subset of $M(\Sigma_\alpha)$, there exist an integer $k\geq3$ and convex open subsets $Q_1,\ldots,Q_k$ of $M(\Sigma_\alpha)$ such that: 
\begin{itemize}

\item[(i)] $(Q_i\cap Q_{i+1})\cap L\neq\emptyset$ for $i=1,\ldots,k-1$;

\item[(ii)] $L\subset\bigcup_{i=1}^k Q_i\subset V_\alpha$ and  $\bigcup_{i=1}^k\psi_\alpha^*(Q_i\cap M(K_\alpha))\subset U$;

\item[(iii)] $\mu_1^+\circ\phi_\alpha^{-1}\in Q_1$ and $\mu_{1}^-\circ\phi_\alpha^{-1}\in Q_k$.

\end{itemize}
%Since $M^{\rm e}(\Sigma_\alpha)$ is  %($G_\delta$ dense)  dense in $M(\Sigma_{\alpha})$(see \cite[Theorem~4]{Sig70}) and $M^{CO}(\Sigma_\alpha)$ is dense in $M^{\rm e}(\Sigma_\alpha)$ \textcolor{red}{REF},
Since $M^{CO}(\Sigma_\alpha)$ is   
dense in $M(\Sigma_{\alpha})$
by \cite[Theorem~1]{Sig70},
 (i) implies
  $(Q_i\cap Q_{i+1})\cap M^{CO}(\Sigma_\alpha)\neq\emptyset$ for $i=1,\ldots,k-1$.
For each $i\in\{1,\ldots,k-1\}$ we fix $\nu_i\in(Q_i\cap Q_{i+1})\cap M^{CO}(\Sigma_\alpha)$ such that 
$\mu_1^+\circ\phi_\alpha^{-1}\neq\nu_1$, 
$\nu_i\neq\nu_{i+1}$ for $i=1,\ldots,k-2$
and $\nu_k\neq\mu_{1}^-\circ\phi_\alpha^{-1}$.
By Lemma~\ref{Sig-lem}, the following statements hold:

\begin{itemize}
\item There is a high complexity path that lies in  
$Q_1\cap M^{\rm e}(\Sigma_\alpha)$ and joins $\mu_1^+\circ\phi_\alpha^{-1}$, $\nu_1$. 

\item For each $i\in\{1,\ldots,k-2\}$, there is a high complexity path that lies in  
$Q_{i+1}\cap M^{\rm e}(\Sigma_\alpha)$ and
joins $\nu_i$, $\nu_{i+1}$.

\item There is a high complexity path that lies in  
$Q_{k}\cap M^{\rm e}(\Sigma_\alpha)$ and joins $\nu_k$,
$\mu_{1}^-\circ\phi_\alpha^{-1}$.

\end{itemize}
 Concatenating all these paths 
yields a path that lies in 
$(\bigcup_{i=1}^k Q_i)\cap M^{\rm e}(\Sigma_\alpha)$ 
and joins $\mu_1^+\circ\phi_\alpha^{-1}$, 
$\mu_{1}^-\circ\phi_\alpha^{-1}$.
From (ii) and Lemma~\ref{include-lem2}(a), this path can be transported %@pulled back
to a high complexity path that lies in
$U\cap M^{\rm e}_{\alpha}(\Sigma_D)$ and joins $\mu_1^+$, $\mu_1^-$. 
\end{proof}

By Lemma~\ref{1-path},
there is a high complexity path $t\in [0,1]\mapsto \mu_t^0\in U\cap M^{\rm e}_{\alpha}(\Sigma_D)$ 
that joins $\mu_1^+$, $\mu_1^-$ by  $\mu_0^0=\mu_1^+$ and $\mu_1^0=\mu_{1}^-$.
We define a map $t\in[0,1]\mapsto\mu_t\in U$ by
\[\mu_t=\begin{cases}\mu^+ &\text{for }t=0,\\
\mu_{n,2^{n+2}(t-2^{-n-2})}^+&\text{for }t\in[2^{-n-2},2^{-n-1}],\ n\in\mathbb N,\\
\mu_{2(t-1/4)}^0 &\text{for }t\in[1/4,3/4],\\
\mu_{n,2^{-n+2}(t-1+2^{n-1})}^-&\text{for }t\in[1-2^{n-1},1-2^{n-2}],\ n\in\mathbb N,\\
\mu^- &\text{for }t=1.\end{cases}\]
The construction gives
$\{\mu_t\colon t\in[2^{-n-2},2^{-n-1}]\}=\{\mu_{n,t}^+\colon t\in[0,1]\}\subset W_n^+$
and
 $\{\mu_t\colon t\in[1-2^{n-1},1-2^{n-2}]\}=\{\mu_{n,t}^-\colon t\in[0,1]\}\subset W_n^-$ for all $n\in\mathbb N$. So,  Lemma~\ref{contained-lem} implies the continuity of $t\mapsto\mu_t$ at $t=0$ and $t=1$. Hence
 this path lies in
 $U\cap (M^{\rm e}_{0}(\Sigma_D)\cup M^{\rm e}_{\alpha}(\Sigma_D))$
 and joins $\mu^+$, $\mu^-$. See \textsc{Figure}~\ref{mu} for a schematic picture.
Since 
it is piecewise homeomorphic, the set
$\{t\in[0,1]\colon\mu_t=\mu\}$ is countable
for all $\mu\in \{\mu_t\colon t\in[0,1]\}$.
Therefore, this path is a high complexity path.
This completes the proof of Proposition~\ref{join-lem}.
\qed

\subsection{More on path connectedness }\label{add-sec}
In addition to Proposition~\ref{cor-path-1}, the following statement is of independent interest.
 
\begin{prop}\label{cor-path-2}
The spaces $M_{\alpha}^{\rm e}(\Sigma_D)$ and $M_{\beta}^{\rm e}(\Sigma_D)$ are path connected and locally path connected.\end{prop}
\begin{proof} Let $\gamma\in\{\alpha,\beta\}$.
Slightly modifying the proof of Proposition~\ref{join-lem}, one can show that for any convex open subset $U$ of $M(\Sigma_D)$ and for any pair of distinct measures in $\mu^+$, $\mu^-\in U\cap M^{\rm e}_{\gamma}(\Sigma_D)$, there exists a high complexity path that lies in $U\cap M^{\rm e}_{\gamma}(\Sigma_D)$ and joins them.
The path connectedness of $M_{\gamma}^{\rm e}(\Sigma_D)$ is a consequence of this with $U=M(\Sigma_D)$.
Since any point in $M(\Sigma_D)$ has a neighborhood base consisting of convex open sets,
the local path connectedness of $M_{\gamma}^{\rm e}(\Sigma_D)$ is also a consequence.
\end{proof}

%The proof of Corollary~\ref{cor-path-2} is analogous to that of Corollary~\ref{cor-path-1} and hence omitted.

\section{Proofs of the main results}
We are almost ready to complete the proofs of the main results of this paper.
In Section~\ref{FA} we collect a few ingredients on functional analysis. In Section~\ref{pfthma} we complete the proof of Theorem~A. In Section~\ref{pfthmb} we complete the proof of Theorem~B.
\subsection{Functional analysis}\label{FA}
The proof of Theorem~A(b) requires the same set of functional analytic ingredients in \cite[Section~3.1]{STY24}. Here we copy them for the reader's convenience.

For a Banach space $V$ with a norm $\|\cdot\|$,
let $V^*$ denote 
the set of real-valued bounded linear functionals on $V$. For each $\mu\in V^*$
let $\|\mu\|$ denote the norm
\[\|\mu\|=\sup\left\{|\mu(f)|\colon f\in V,\ \|f\|=1\right\}.\] Let $\Lambda$, $\mu\in V^*$. We say:
\begin{itemize}

\item 
$\mu$ is {\it tangent} to $\Lambda$ at $f\in V$ if $\mu(f)\leq\Lambda(f+g)-\Lambda(f)$ holds for all $g\in V$.

\item 
 $\mu$ is {\it bounded} by $\Lambda$ if $\mu(f)\leq\Lambda(f)$  holds for all $f\in V$.

\item 
 $\Lambda$ is {\it convex} if 
 $\Lambda(tf+(1-t)g)\leq t\Lambda(f)+(1-t)\Lambda(g)$ holds for all $f$, $g\in V$ and $t\in[0,1]$.
\end{itemize}

%The following is known as Bishop-Phelps's theorem.
\begin{thm}[\cite{Isr79}, Theorem~V.1.1]\label{bis}
Let $V$ be a Banach space and let $\Lambda\in V^*$ be convex and continuous.
%Let $\Delta$ %\colon V\to \mathbb R$ be a convex continuous functional on a Banach space $V$. 
For any $\mu\in V^*$ that is bounded by $\Lambda$, any $f\in V$ and any $\varepsilon>0$, there exist $\tilde\mu\in V^*$ and $\tilde f\in V$ such that $\tilde\mu$ is tangent to $\Lambda$ at $\tilde f$ and
\[\|\tilde\mu-\mu\|\leq\varepsilon\ \text{ and }\
\|\tilde f-f\|\leq\frac{1}{\varepsilon}(\Lambda(f)-\mu(f)+s),\]
where $s=\sup\{\mu(g)-\Lambda(g)\colon g\in V\}\leq0$.
\end{thm}
% For a (finite) signed Borel measure $\mu$ on a compact metric space $X$, let
% $\mu(f)$ denote the integral of $f\in C(X)$ against $\mu$.
For a continuous map $T$ of a compact metric space $X$, 
the functional $\Lambda_T$ on $C(X)$ given by \eqref{alphaT}
%\begin{equation}\label{lambdat}\Lambda_T(f)=\max\left\{\mu(f)\colon\mu\in M(X,T)\right\}.\end{equation}
%Clearly $\Lambda_T$ 
is convex and continuous.
The next lemma characterizes maximizing measures in terms of $\Lambda_T$. %For an interpretation, 
Recall that 
 $C(X)^*$ can be identified with the set of (finite) signed Borel measures on $X$ by
 Riesz's representation theorem.

\begin{lemma}[\cite{Bre08}, Lemma~2.3]\label{Bre}
Let $T$ be a continuous map of a compact metric space $X$ and
let $f\in C(X)$. Then $\mu\in C(X)^*$ is tangent to $\Lambda_T$ at $f$ if and only if
$\mu$ belongs to $M(X,T)$ and is $f$-maximizing.
\end{lemma}

%Let $X$ be a topological space. 
%For a (finite) signed Borel measure $\mu$ on a compact metric space $X$, let $\|\mu\|$ denote the norm of $\mu$ as a bounded liner functional on $C(X)$, namely
%\[\|\mu\|=\sup\left\{\left|\int f{\rm d}\mu\right|\colon f\in C(X),\ \|f\|_{C^0}=1\right\}.\]
%The following lemma is a special case of \cite[Appendix~5]{Rue04}. For completeness we include a proof in Appendix~\ref{A2}.
If $T$ is a continuous map of $X$, then
for any $\mu\in M(X,T)$ there exists a unique Borel probability measure $b_\mu$ on $M(X,T)$ such that $b_\mu(M^{\rm e}(X,T))=1$ and
$\mu=\int_{M^{\rm e}(X,T)} \nu{\rm d}b_\mu(\nu)$. We call $b_\mu$ the barycenter of $\mu$. Put
\[{\rm supp}(b_\mu)
=\bigcap\{F\colon\text{$F\subset M(X,T)$, closed, $b_\mu(F)=1$}\}.\]
Since $M(X,T)$ has a countable base, we have $b_\mu({\rm supp}(b_\mu))=1$.
\begin{lemma}[\cite{STY24}, Lemma~3.3]\label{supp-lem}
Let $T$ be a continuous map of a 
compact metric space $X$. 

\begin{itemize}

\item[(a)]  If there exists a constant $C\geq0$ such that $h(\nu)\leq C$ for all $\nu\in {\rm supp}(b_\mu)$, then $h(\mu)\leq C$.

\item[(b)] Let $f\in C(X)$, $\mu\in M_{\rm max}(f)$. Then ${\rm supp}(b_\mu)$ is contained in $M_{\rm max}(f)$.
\end{itemize}
 \end{lemma}
The next lemma asserts that the barycenter map $\mu\mapsto b_\mu$
from $M(X,T)$ to the set of Borel probability measures
on $M(X,T)$ 
is isometric.
\begin{lemma}[\cite{Isr79}, Corollary~IV.4.2]\label{norm-lem}
Let $T$ be a continuous map of a 
compact metric space $X$.
For all $\mu$, $\mu'\in M(X,T)$ 
we have \[\|b_{\mu}-b_{\mu'}\|=\|\mu-\mu'\|.\]
\end{lemma}

\subsection{Proof of Theorem~A }\label{pfthma}
For each $n\in\mathbb N$, define
\[O_n=\left\{\mu\in \overline{M^{\rm e}(\Sigma_D)}: 0\leq h(\mu)<\frac{1}{n}\right\},\]
and
\[
    U_n=\{f\in C(\Sigma_D)\colon \overline{M^{\rm e}(\Sigma_D)}\cap M_{{\rm max}}(f)\subset O_n\}.\]
        Clearly, the set
$\mathscr{R}=\{f\in C(\Sigma_D)\colon
h(\mu)=0\ \text{ for all } \mu\in M_{\rm max}(f)\}$ is contained in $\bigcap_{n= 1}^\infty U_n.$ Conversely, let 
$f\in\bigcap_{n=1}^\infty U_n$.
Any ergodic measure in $M_{\rm max}(f)$ has zero entropy, and by Lemma~\ref{supp-lem}, any non-ergodic measure in $M_{\rm max}(f)$ has zero entropy too.
Hence we obtain 
$\mathscr{R}=\bigcap_{n= 1}^\infty U_n.$

By the upper semicontinuity of the entropy function, $O_n$ is an open subset of $\overline{M^{\rm e}(\Sigma_D)}$.
$CO$-measures have zero entropy, and they are dense in $M^{\rm e}(\Sigma_D)$ by Proposition~\ref{per-approx}(a). It follows that $O_n$ is a dense subset of $\overline{M^{\rm e}(\Sigma_D)}$. 
By the result of Morris \cite[Theorem~1.1]{Mor10}, $U_n$
is an open and dense subset of $C(\Sigma_D)$.
Therefore, $\mathscr{R}$ is dense $G_\delta$ as required in Theorem~A(a). 

To prove Theorem~A(b), let $f\in C(\Sigma_D)$.
By Proposition~\ref{join-lem}, there is a high complexity path 
$t\in[0,1]\mapsto\mu_t\in M^{\rm e}(\Sigma_D)$
such that $\mu_0\in M_{\rm max}(f)$. 
 Let $\varepsilon\in(0,1/2)$ and put
\[R=\left\{\nu\in M(\Sigma_D)\colon  \int f{\rm d}\nu\geq\Lambda_\sigma(f)-\varepsilon^2\right\}.\]
%where
%\[\Lambda_\sigma(f_0)=\max\left\{\mu(f_0)\colon\mu\in M(\Sigma_D,\sigma)\right\},\]
%see \eqref{lambdat}.
Let $m$ denote the Lebesgue measure on $[0,1]$, and define a Borel probability measure $\hat m$ on $M^{\rm e}(\Sigma_D)$ by 
$\hat m(\cdot)=m\{t\in[0,1]\colon\mu_t\in\cdot\}$.
Since $f\in C(\Sigma_D)$, we have
 $\hat m(R\cap M^{\rm e}(\Sigma_D))>0.$
Let $\hat m_R$ denote the normalized restriction of $\hat m$ to $R\cap M^{\rm e}(\Sigma_D)$, and
put $\mu=\int_{M^{\rm e}(\Sigma_D)}\nu{\rm d}\hat m_R(\nu)$. Then $\mu$ 
belongs to $R$ and is bounded by $\Lambda_\sigma$ as an element of $C(\Sigma_D)^*$.
Note that $b_{\mu}=\hat m_R$.
By Theorem~\ref{bis}, there exist $\tilde f\in C(\Sigma_D)$ and $\tilde\mu\in C(\Sigma_D)^*$ such that $\tilde\mu$ is tangent to $\Lambda_\sigma$ at $\tilde f$, and
\begin{equation}\label{bis-eq}\|\tilde\mu-\mu\|\leq\varepsilon\ \text{ and }\ \|\tilde f-f\|_{C^0}\leq\frac{1}{\varepsilon}\left(\Lambda_\sigma(f)-\int f{\rm d}\mu\right)\leq\varepsilon.\end{equation}
By Lemma~\ref{Bre}, $\tilde\mu$ is shift-invariant and $\tilde f$-maximizing.

Take %\textcolor{red}{{\sout{a proper}}} 
an open subset $U$ of $M(\Sigma_D)$ such that 
${\rm supp}(b_{\tilde\mu})\subset U$
and $b_{\mu}(U\setminus{\rm supp}(b_{\tilde\mu}))<\varepsilon$.
Since $M^{\rm e}(\Sigma_D)$ is a metric space, it is a normal space. 
By Urysohn's lemma, there exists $g\in C(M(\Sigma_D))$ such that $\|g\|_{C^0}=1$, 
$g\equiv0$ on $M(\Sigma_D)\setminus U$ and
$g\equiv1$ on ${\rm supp}(b_{\tilde\mu})$.
We have
\[\begin{split}b_{\mu}({\rm supp}(b_{\tilde\mu}))&>
b_{\mu}(U)-\varepsilon>b_{\mu}(g)-\varepsilon\geq b_{\tilde\mu}(g)-2\varepsilon\\
&\geq b_{\tilde\mu}({\rm supp}(b_{\tilde\mu}))-2\varepsilon=1-2\varepsilon>0.\end{split}\]
To deduce the third inequality, we have used
$\|b_{\tilde\mu}-b_{\mu}\|=\|\tilde\mu-\mu\|\leq\varepsilon$
from Lemma~\ref{norm-lem} and the first inequality in \eqref{bis-eq}. 
Since $b_{\mu}$ is non-atomic from the property (A2), 
it follows that
the set $\{\mu_t\colon t\in[0,1],\ \mu_t\in {\rm supp}(b_{\tilde \mu})\}$
contains uncountably many elements, which belong to $M_{\rm max}(\tilde f)$ by Lemma~\ref{supp-lem}.
Since $f\in C(\Sigma_D)$ and $\varepsilon\in(0,1/2)$ are arbitrary,
the proof of Theorem~A(b) is complete.\qed

\subsection{Proof of Theorem~B}\label{pfthmb}

Since $M^{\rm e}(\Sigma_D)=M^{\rm e}_{\alpha}(\Sigma_D)\cup M^{\rm e}_{\beta}(\Sigma_D)\cup  M^{\rm e}_{0}(\Sigma_D)$, the path connectedness of 
$M^{\rm e}(\Sigma_D)$ follows from Proposition~\ref{cor-path-1}.
Let $U$
be a convex open subset of $M(\Sigma_D)$. Since $M^{\rm e}(\Sigma_D)$ is dense in $M(\Sigma_D)$,  we have $U\cap M^{\rm e}(\Sigma_D)\neq\emptyset$.
By Proposition~\ref{join-lem}, for each $\gamma\in\{\alpha,\beta\}$ the set
$U\cap(M^{\rm e}_{0}(\Sigma_D)\cup M^{\rm e}_{\gamma}(\Sigma_D))$  
is path connected unless empty.
It follows that $U\cap M^{\rm e}(\Sigma_D)$ 
is path connected.
Since any point in $M(\Sigma_D)$ has a neighborhood base consisting of convex open sets and
 $U$ is an arbitrary convex open subset of $M(\Sigma_D)$, we have verified that
$M^{\rm e}(\Sigma_D)$ is locally path connected.
The proof of Theorem~B is complete.
\qed

\subsection*{Acknowledgments}
Part of this paper was written while HT and KY were visiting Kumamoto University. We thank Naoya Sumi for his hospitality during our visit and for fruitful discussions.
MS was supported by the JSPS KAKENHI 21K13816. HT was supported by the JSPS KAKENHI 23K20220.
KY was supported by the JSPS KAKENHI
21K03321.
      \bibliographystyle{amsplain}

\end{document}